\documentclass[11pt]{article}

\title{Asymptotic Automorphism Groups of Circulant Graphs and Digraphs}

\author{Soumya Bhoumik and Edward Dobson\\
Department of Mathematics and Statistics \\
Mississippi State University\\
Mississippi State, MS 39759 USA\\
and\\
Joy Morris\\
Department of Mathematics and Computer Science\\
University of Lethbridge\\
Lethbridge, AB T1K 3M4 Canada
}

\usepackage{amssymb}

\usepackage{theorem}

\usepackage{cite}

\newtheorem{thrm}{Theorem}[section]

\newtheorem{lem}[thrm]{Lemma}
\newtheorem{cor}[thrm]{Corollary}

\newtheorem{conj}[thrm]{Conjecture}
\usepackage{color}
\usepackage{pdfsync}

\theorembodyfont{\rm} \theoremstyle{definition}
\newtheorem{defin}[thrm]{Definition}

 \newcounter{case}

 \renewcommand{\thecase}{\arabic{case}}

\newcounter{subcase}

 \renewcommand{\thesubcase}{\alph{subcase}}

\usepackage{fullpage}

\def\tl{\triangleleft}
\def\AGL{{\rm AGL}}

\def\mod{{\rm mod\ }}

\def\fix{{\rm fix}}
\def\Aut{{\rm Aut}}
\def\Stab{{\rm Stab}}
\def\gcd{{\rm gcd}}
\def\la{\langle}
\def\ra{\rangle}
\def\Z{{\mathbb Z}}
\def\N{{\mathbb N}}

\def\AGL{{\rm AGL}}

\def\SDW{{\rm SDW}}
\def\NonNor{{\rm NonNor}}
\def\NonNorG{{\rm NonNorG}}
\def\DW{{\rm DW}}
\def\SW{{\rm GW}}

\def\H\times S_m{{\rm H\times S_m}}
\def\Nor{{\rm Nor}}
\def\NorG{{\rm NorG}}
\def\DRR{{\rm DRR}}
\def\ACD{{\rm ACD}}
\def\ACG{{\rm ACG}}
\def\Small{{\rm Small}}
\def\SWG{{\rm GWG}}
\def\DWG{{\rm DWG}}

\newenvironment{proof}{\noindent {\sc Proof}.}
                {\phantom{a} \hfill \framebox[2.2mm]{ } \bigskip}

\begin{document}

\pagestyle{plain}

\baselineskip = 1.3\normalbaselineskip

\maketitle

\begin{abstract}

We show that almost all circulant graphs have automorphism groups as small as possible. Of the circulant graphs
that do not have automorphism group as small as possible, we give some families of integers such that it is not true that almost all circulant graphs whose order lies in any one of these families, are normal.  That almost all Cayley (di)graphs whose automorphism group is not as small as possible are normal was conjectured by the second author, so these results provide counterexamples to this conjecture.  It is then shown that there is a ``large" family of integers for which almost every circulant digraph whose order lies in this family and that does not have automorphism group as small as possible, is normal.  We additionally explore the asymptotic behavior of the automorphism groups of circulant (di)graphs that are not normal, and show that no general conclusion can be obtained.
\end{abstract}

\section{Introduction}

Determining the full automorphism group of a Cayley (di)graph is one of the most fundamental questions one can ask about a Cayley (di)graph. While it is usually quite difficult to determine the automorphism group of a Cayley (di)graph, characterizing almost all Cayley graphs of a group $G$, based on the structure of $G$, has been of consistent interest in the last few decades. Babai, Godsil, Imrich, and Lov\'asz (see \cite[Conjecture 2.1]{BabaiG1982}) conjectured that almost all Cayley graphs of any group $G$ that is not generalized dicyclic or abelian with exponent greater than $2$ are GRR's (have automorphism group $G_L$, the left regular representation of $G$).  A similar conjecture was made for digraphs (with no exceptions) by Babai and Godsil \cite{BabaiG1982}.  Babai and Godsil \cite[Theorem 2.2]{BabaiG1982} proved these two conjectures for nilpotent (and nonabelian) groups of odd order. In 1998, Xu \cite{Xu1998} introduced the notion of a normal Cayley (di)graph of a group $G$:  a Cayley (di)graph $\Gamma$ of a group $G$ such that $G_L\lhd \Aut(\Gamma)$.  Xu also conjectured that for each positive integer $n$ there exists a group $G$ such that almost all Cayley (di)graphs of $G$ are normal Cayley (di)graphs of $G$ (see \cite[Conjecture 1]{Xu1998} for the precise formulation of this conjecture). In 2010, the second author showed that almost all Cayley graphs of an abelian group $G$ of odd prime-power order are normal \cite{Dobson2010b}.

In this paper, we first investigate in Section 3 the proportion of the set of normal circulant (di)graphs in the family of circulant (di)graphs. We show that almost all circulant graphs have automorphism group as small as possible (Theorem \ref{graphsmall}), which make them normal immediately. In \cite[Conjecture 4.1]{Dobson2010b}, the second author conjectured that almost every Cayley (di)graph whose automorphism group is not as small as possible is a normal Cayley (di)graph.  We show that this conjecture fails for circulant digraphs of order $n$ (Theorem \ref{conjfalse1}), where $n\equiv 2\ (\mod 4)$ has a fixed number of distinct prime factors, and point out some ``gaps" in the proof of \cite[Theorem 3.5]{Dobson2010b}, which leads to additional counterexamples to \cite[Conjecture 4.1]{Dobson2010b} for graphs in the case where $n = p$ or $p^2$ and $p$ is a {\bf safe prime}, i.e. $p = 2q + 1$ where $q$ is prime, or when $n$ is a power of 3 (Theorem \ref{primepower}).  Finally, we prove that the conjecture holds for digraphs of order $n$ where $n$ is odd and not divisible by $9$ (Theorem \ref{normaldi}). We also show that the conjecture holds for graphs of order $n$ where $n$ is still odd and not divisible by 9, if we add the extra condition that $n$ is not of the form $n = p$ or $p^2$ where $p$ is a safe prime (Theorem \ref{normal}).

In Section 4, we focus on non-normal circulant (di)graphs. A variety of authors have shown that non-normal Cayley (di)graphs are either generalized wreath products (see Definition \ref{semiwreathdefin}) or have automorphism group that
of a deleted wreath product (see Definition \ref{dwdefin}). We show that there exist sets of integers $S_1,S_2$, and a family of sets of integers $S_c$ such that almost all non-normal circulant graphs and digraphs whose order is in $S_1$ have automorphism group that of a deleted wreath product (Theorem \ref{n=pq^2r^2}), almost all non-normal circulant graphs and digraphs whose order is in $S_2$ are generalized wreath products (Theorem \ref{p^2|n}), and neither generalized wreath products nor those graphs whose automorphism group is that of a deleted wreath product of circulant graphs and digraphs dominates amongst those whose order is in any $S_c$ (Theorem \ref{counterexample}).  We remark that we do not know if any set $S_c$ is infinite (but when $c=2$ for example, $S_c$ consists of all products of twin primes).

In the next section, we will focus on background results and terminology, as well as developing the counting tools needed in Sections 3 and 4.

\section{Preliminaries and tools}

We start by stating basic definitions, and then proceed to known results in the literature that we will need.  We will finish with results that will be the main tools throughout the rest of the paper.

By ``almost all" circulant (di)graphs in some family $F_1$ of circulant (di)graphs of order in a set $S$ of integers being in some family $F_2$ of circulant (di)graphs, we mean that

$$\lim_{n\in S,n\to\infty}\frac{\vert F_2\vert}{\vert F_1\vert} = 1.$$

\begin{defin}\label{Cayleygraph}
Let $G$ be a group and $S\subset G$ such that $1_G\not\in S$. Define
a digraph $\Gamma = \Gamma(G,S)$ by $V(\Gamma) = G$ and $E(\Gamma) =
\{(u,v):v^{-1}u\in S\}$.  Such a digraph is a {\it Cayley digraph of $G$
with connection set $S$}. A Cayley graph of $G$ is defined
analogously though we insist that $S = S^{-1} = \{s^{-1}:s\in S\}$.  If $G$ is a cyclic group, then a Cayley (di)graph of $G$ is a {\it circulant (di)graph of order $n$}, where $\vert G\vert = n$.
\end{defin}

It is straightforward to verify that for $g\in G$, the map $g_L:G\to
G$ by $g_L(x) = gx$ is an automorphism of $\Gamma$. Thus $G_L =
\{g_L:g\in G\}$, the left regular representation of $G$, is a
subgroup of the automorphism group of $\Gamma$, $\Aut(\Gamma)$.

\begin{defin}
Let $G$ be a transitive permutation group with complete block system ${\cal B}$.  By $G/{\cal B}$, we mean the subgroup of
$S_{\cal B}$ induced by the action of $G$ on ${\cal B}$, and by $\fix_G({\cal B})$ the kernel of this action.
Thus $G/{\cal B} = \{g/{\cal B}:g\in G\}$ where $g/{\cal B}(B_1) = B_2$ if and only if $g(B_1) = B_2$, $B_1,B_2\in{\cal B}$, and $\fix_G({\cal B}) = \{g\in G:g(B) = B{\rm\ for\ all\ }B\in{\cal B}\}$.
\end{defin}

It is not difficult to show using the fact that a transitive abelian group is regular \cite[Proposition 4.4]{Wielandt1964}, and that every block system of a permutation group $G$ containing a regular abelian subgroup is formed by the orbits of a normal subgroup of $G$ (in fact, formed by the orbits of a subgroup of a regular
abelian subgroup of $G$).  In this paper, the transitive permutation groups that we will encounter will also contain a regular cyclic subgroup, and so every complete block system will be formed by the orbits of a normal subgroup.  In fact, a complete block system will always consist of the cosets of a cyclic group \cite[Exercise 6.5]{Wielandt1964}.

A {\it
vertex-transitive (di)graph} is a (di)graph whose automorphism group acts
transitively on the vertices of the (di)graph.

\begin{defin}
Let $\Gamma_1$ and $\Gamma_2$ be vertex-transitive digraphs.  Let
$$E = \{((x,x'),(y,y')): xy\in E(\Gamma_1), x',y'\in V(\Gamma_2)\mbox{ or $x = y$ and }x'y'\in E(\Gamma_2)\}.$$
Define the {\it wreath} ({\it or lexicographic}) {\it product} of
$\Gamma_1$ and $\Gamma_2$, denoted $\Gamma_1\wr \Gamma_2$, to be the
digraph such that $V(\Gamma_1\wr \Gamma_2) = V(\Gamma_1)\times
V(\Gamma_2)$ and $E(\Gamma_1\wr \Gamma_2) = E$.
\end{defin}

We remark that the
wreath product of a circulant digraph of order $m$ and a circulant
digraph of order $n$ is circulant.
Note that what we have just defined as $\Gamma_1 \wr \Gamma_2$ is sometimes defined as $\Gamma_2 \wr \Gamma_1$, particularly in the work of Praeger, Li, and others from the University of Western Australia.

\begin{defin}
Let $\Omega$ be a set and $G\le S_\Omega$ be transitive.  Let $G$
act on $\Omega\times\Omega$ by $g(\omega_1,\omega_2) =
(g(\omega_1),g(\omega_2))$ for every $g\in G$ and
$\omega_1,\omega_2\in\Omega$.  We define the {\it $2$-closure of
$G$}, denoted $G^{(2)}$, to be the largest subgroup of $S_\Omega$
whose orbits on $\Omega\times\Omega$ are the same as $G$'s.  Let
${\mathcal O}_1,\ldots,{\mathcal O}_r$ be the orbits of $G$ acting
on $\Omega\times\Omega$.  Define digraphs $\Gamma_1,\ldots,\Gamma_r$
by $V(\Gamma_i) = \Omega$ and $E(\Gamma_i) = {\mathcal O}_i$.  Each
$\Gamma_i$, $1\le i\le r$, is an {\it orbital digraph of G}, and it
is straightforward to show that $G^{(2)} =
\cap_{i=1}^r\Aut(\Gamma_i)$.   A {\it generalized orbital digraph of
$G$} is an arc-disjoint union of orbital digraphs of $G$.
\end{defin}

Clearly the automorphism
group of a graph or digraph is $2$-closed.

The following theorem appears in \cite {Li2005} and is a translation of results that were proven in
\cite{EvdokimovP2002, LeungM1998, LeungM1996} using Schur
rings, into group theoretic language. We have re-worded part (1) slightly to clarify the meaning. In the special case of
circulant digraphs of square-free order $n$, an equivalent
result was proven independently in \cite{DobsonM2005}.

\begin{thrm}\label{maintool}
Let $G\le S_n$ contain a regular cyclic subgroup $\la\rho\ra$.  Then
one of the following statements holds:
\begin{enumerate}
\item There exist $G_1, \ldots, G_r$ such that $G^{(2)}=G_1\times\ldots \times G_r$, and for each $G_i$, either $G_i \cong S_{n_i}$, or
 $G_i$ contains a normal regular
cyclic group of order $n_i$. Furthermore, $r\ge
1$, $\gcd(n_i,n_j) = 1$ for $i \neq j$, and $n =
n_1n_2\cdots n_r$.
\item $G$ has a normal subgroup $M$ whose orbits form the complete
block system ${\cal B}$ of $G$ such that each connected generalized
orbital digraph contains a subdigraph $\Gamma$ which is an orbital
digraph of $G$ and has the form $\Gamma = (\Gamma/{\cal B})\wr
\bar{K}_b$, where $b = \vert M\cap \la\rho\ra\vert$.
\end{enumerate}
\end{thrm}

\begin{defin}\label{semiwreathdefin}
A circulant digraph $\Gamma$ with connection set $S$ is said to be a {\it $(K,H)$-generalized wreath circulant digraph} (or just a {\it generalized wreath circulant digraph}) if there exist groups $H$, $K$ with $1<K\le H \le \Z_n$ such that $S\setminus H$ is a union of cosets of $K$.
\end{defin}

The name generalized wreath is chosen for these digraphs as if $K = H$, then $\Gamma$ is in fact a wreath product.  We now wish to investigate the relationship between generalized wreath circulant digraphs and the preceding result.  We shall have need of the following result.

\begin{lem}\label{disconnectedblocks}
Let $\Gamma$ be a disconnected generalized orbital digraph of a
transitive group $G$.  Then the components of $\Gamma$ form a
complete block system ${\cal B}$ of $G$.
\end{lem}

\begin{proof}
As the blocks of $G^{(2)}$ are identical to the blocks of $G$
\cite[Theorem 4.11]{Wielandt1969} (\cite{Wielandt1969} is contained in the more
accessible \cite{Wielandt1994}), it suffices to show that the set of
components ${\cal B}$ of $\Gamma$ is a complete block system of
$G^{(2)}$.  This is almost immediate as $G^{(2)} =
\cap_{i=1}^r\Aut(\Gamma_i)$, where $\Gamma_1,\cdots,\Gamma_r$ are
all of the orbital digraphs of $G$.  Assume that $\Gamma = \cup_{i =
1}^s\Gamma_i$, for some $s\le r$.  Then
$\cap_{i=1}^s\Aut(\Gamma_i)\le\Aut(\Gamma)$, so that ${\cal B}$ is a
complete block system of $\cap_{i=1}^s\Aut(\Gamma_i)$.  Also, $G \le
G^{(2)} = \cap_{i=1}^r\Aut(\Gamma_i)\le\cap_{i=1}^s\Aut(\Gamma_i)$.
Thus ${\cal B}$ is a complete block system of $G^{(2)}$ as ${\cal
B}$ is a complete block system of $\cap_{i=1}^s\Aut(\Gamma_i)$.
\end{proof}

We will require the following partial order on complete block systems.
\begin{defin}\label{partialorderdefin}
We say that ${\cal B}\preceq{\cal C}$ if for every $B\in{\cal B}$ there exists $C\in {\cal C}$ with $B\subseteq C$.  That is, each block of ${\cal C}$ is a union of blocks of ${\cal B}$.
\end{defin}

Our main tool in examining generalized wreath circulants will be the following result.

\begin{lem}\label{semiwreath}
Let $G$ be $2$-closed with a normal subgroup $M$ and a regular subgroup $\la \rho\ra$.  Let $\cal B$ be the complete block system of $G$ formed by the orbits of $M$, and suppose that
each connected generalized orbital digraph contains
a subdigraph $\Gamma$ which is an orbital digraph of $G$ and has
the form $\Gamma = (\Gamma/{\cal B})\wr \bar{K}_b$, where $b = \vert
M\cap \la\rho\ra\vert$.  Then there exists a complete block system ${\cal C}\succeq{\cal B}$
of $G$ such that $\fix_{G^{(2)}}({\cal B})\vert_C\le G^{(2)}$ for
every $C\in{\cal C}$.
\end{lem}

\begin{proof}
Observe that we may choose $M = \fix_G({\cal
B})$, in which case $\vert M\cap \la\rho\ra\vert = \vert B\vert$,
where $B\in{\cal B}$, so that $b$ is the size of a block of ${\cal
B}$.  First suppose that  if $B,B'\in{\cal B}$, $B\not = B'$, then
any orbital digraph $\Gamma'$ that contains some edge of the form
$\vec{xy}$ with $x\in B$, $y\in B'$ has every edge of the form
$\vec{xy}$, with $x\in B$, $y\in B'$. It is then not difficult to
see that every orbital digraph $\Gamma$ of $G$ can be written as a
wreath product $\Gamma' = \Gamma_1\wr\Gamma_2$, where $\Gamma_1$ is
a circulant digraph of order $n/b$ and $\Gamma_2$ is a circulant
digraph of order $b$. Then $G/{\cal B}\wr \fix_G({\cal B})\vert_B\le
\Aut(\Gamma')$ for every orbital digraph $\Gamma'$, and so $G/{\cal
B}\wr\fix_G({\cal B})\vert_B\le G^{(2)}$. Then result then follows with
${\cal C} = {\cal B}$.  (Note that $G$ is 2-closed, so $G^{(2)}=G$.)

For convenience, we denote the orbital digraph that contains the
edge $\vec{xy}$ by $\Gamma_{xy}$.  We may now assume that there
exists some $B,B'\in{\cal B}$, $B\not = B'$, and $x\in B$, $y\in B'$
such that $\Gamma_{xy}$ does not have every edge of the form
$\vec{x'y'}$, with $x'\in B$ and $y'\in B'$. Note then that no
$\Gamma_{x'y'}$ with $x'\in B$ and $y'\in B'$ has every directed
edge from $B$ to $B'$. Let ${\cal X}$ be the set of all
$\Gamma_{xy}$ such that if $x\in B_1\in{\cal B}$ and $y\in B_2\in{\cal
B}$, $B_1\not = B_2$ then $\Gamma_{xy}$ does not have every edge from
$B_1$ to $B_2$.  Let $\hat{\Gamma}$ be the generalized orbital digraph
whose edges consist of all edges from every orbital digraph in
${\cal X}$, as well as every directed edge contained within a block
of ${\cal B}$. Then no orbital digraph that is a subgraph of $\hat{\Gamma}$
can be written as a connected wreath product $\Gamma'\wr\bar{K}_b$
for some $\Gamma'$, and so by hypothesis, $\hat{\Gamma}$ must be
disconnected.

By Lemma \ref{disconnectedblocks}, the components of
$\hat{\Gamma}$ form a complete block system ${\cal C}\succeq{\cal
B}$ of $G$.  (To see that $\cal C \succeq \cal B$, note that $\hat\Gamma$ contains every edge from $B$ to $B'$, so $B$ is in a connected component of $\hat\Gamma$.  Since $G$ is transitive ($\la \rho\ra \le G$), $\cal C \succeq \cal B$.) Let $\Gamma_1,\Gamma_2,\ldots,\Gamma_r$ be the orbital
digraphs of $G$, and assume that $\cup_{i=1}^s\Gamma_i = \hat{\Gamma}$.  If
$1\le i\le s$, then $(G^{(2)}/{\cal C})\wr\fix_{G^{(2)}}({\cal
C})\le\Aut(\Gamma_i)$ as $G^{(2)}\le\Aut(\Gamma_i)$, $\Gamma_i$ is
disconnected, and each component is contained in a block of ${\cal
C}$.  Thus $\fix_{G^{(2)}}({\cal B})\vert_C\le\Aut(\Gamma_i)$ for
every $1\le i\le s$.  If $s+1\le i\le r$, then if $B,B'\in{\cal B}$,
$B\not = B'$ and $\vec{xy}\in E(\Gamma_i)$ for some $x\in B$, $y\in
B'$, then $\vec{xy}\in E(\Gamma_i)$ for every $x\in B$ and $y\in
B'$.  Also observe that as the subgraph of $\hat{\Gamma}$ induced by $B$ is $K_b$, the subgraph of $\Gamma_i$ induced by $G$ is $\bar{K}_b$.  We conclude that $\Gamma_i$ is the wreath product $\Gamma_i/{\cal B}\wr \bar{K}_b$, and so $\Aut(\Gamma_i)$ contains
$\Aut(\Gamma_i/{\cal B})\wr S_b$.  Then $\fix_{G^{(2)}}({\cal
B})\vert_B\le\Aut(\Gamma_i)$ for every $B\in{\cal B}$.  As ${\cal
B}\preceq{\cal C}$,  $\fix_{G^{(2)}}({\cal
B})\vert_C\le\Aut(\Gamma_i)$ for every $1\le i\le r$ and as $G^{(2)}
= \cap_{i=1}^r\Aut(\Gamma_i)$, $\fix_{G^{(2)}}({\cal B})\vert_C\le
G^{(2)}$ for every $C\in{\cal C}$.
\end{proof}

\begin{lem}\label{semiwreathequiv}
Let $\Gamma$ be a circulant digraph of order $n$. Then $\Gamma$ is a $(K,H)$-generalized wreath circulant digraph if and only if there exists $G\le\Aut(\Gamma)$ such that $G$ contains a regular cyclic subgroup, and $\fix_{G^{(2)}}({\cal B})\vert_C\le G^{(2)}$ for every $C\in{\cal C}$, where ${\cal B}\preceq{\cal C}$ are formed by the orbits of $K$ and $H$, respectively.
\end{lem}

\begin{proof}
Suppose first that $G \le \Aut(\Gamma)$ with $\rho$ a generator of a regular cyclic subgroup in $G$, and there exist complete block systems ${\cal B}\preceq{\cal C}$
of $G$ such that $\fix_{G^{(2)}}({\cal B})\vert_C\le G^{(2)} \le \Aut(\Gamma)$ for
every $C\in{\cal C}$.  Since $\rho \in G$, the action of $\fix_{G^{(2)}}({\cal B})\vert_C$ is transitive on every $B \subseteq C$, so between any two blocks $B_1,B_2\in{\cal B}$ that are not contained in a block of ${\cal C}$, we have that there is either every edge from $B_1$ to $B_2$ or no edges from $B_1$ to $B_2$.   Let $\cal B$ be formed by the orbits of $K\le \la \rho\ra$. Then for every edge $\vec{xy}$ whose endpoints are not both contained within a block of ${\cal C}$, $(y - x) + K \subset S$.  Let ${\cal C}$ be formed by the orbits of $H\le \la \rho\ra$.  Then $S \setminus H$ is a union of cosets of $K$ as required.

Conversely, if $\Gamma$ is a $(K,H)$-generalized wreath circulant, then it is not hard to see that $\rho^m \vert_C \in \Aut(\Gamma)$ for every $C \in \cal C$, where $\rho$ generates $(\Z_n)_L$ and $m=[\Z_n : K]$.  Let $G$ be the largest subgroup of $\Aut(\Gamma)$ that admits both ${\cal B}$ and ${\cal C}$ as complete block systems; clearly $\rho \in G$.  Also,  since $G^{(2)}$ has the same block systems as $G$ and is a subgroup of $\Aut(\Gamma)$, $G^{(2)}=G$. Now, if $g\in\fix_G({\cal B})$, then $g\vert_C\in \Aut(\Gamma)$ as well.  But this implies that $g\vert_C\in G$ and the result follows.
\end{proof}

Combining Lemma \ref{semiwreath} and Lemma \ref{semiwreathequiv}, and recalling that the full automorphism group of a (di)graph is always 2-closed, we have the following result.

\begin{cor}\label{semiwreathchar}
Let $\Gamma$ be a circulant digraph whose automorphism group $G = \Aut(\Gamma)$ satisfies Theorem \ref{maintool} (2).  Then $\Gamma$ is a generalized wreath circulant digraph.
\end{cor}

We now wish to count the number of generalized wreath circulant digraphs.

\begin{lem}\label{numbersemiwreath}
The total number of generalized wreath circulant digraphs of order $n$ is at
most
$$\displaystyle\sum_{p|n}2^{n/p-1}\biggr(\displaystyle\sum_{q|(n/p)}2^{(n-n/p)/q}\biggr),$$ where $p$ and $q$ are prime.
\end{lem}

\begin{proof}
Let $\Gamma$ be a $(K,H)$-generalized wreath circulant digraph of order $n$. By Lemma \ref{semiwreathequiv},
there exists $G\le\Aut(\Gamma)$ that admits ${\cal B}$ and ${\cal C}$ such that $\rho \in G$, and
$\fix_{G^{(2)}}({\cal B})\vert_C\le\Aut(\Gamma)$ for every
$C\in{\cal C}$, where ${\cal B}$ is formed by the orbits of $K$ and ${\cal C}$ is formed by the orbits of $H$ (so $\cal B$ consists of the cosets of $K$ and $\cal C$ consists of the cosets of $H$).  Let ${\cal B}$ consist of $m$ blocks of size $k$.
Then $\rho^m\vert_C\in\Aut(\Gamma)$ for every $C\in{\cal C}$. Choose
$q\vert k$ to be prime, and let $G'\le\Aut(\Gamma)$ be the largest
subgroup of $\Aut(\Gamma)$ that admits a complete block system
${\cal D}$ consisting of $n/q$ blocks of size $q$.  Note then that
$\rho^{n/q}\vert_C\in G'$ for every $C\in{\cal C}$.  Let $p$ be a
prime divisor of the number of blocks of ${\cal C}$, and ${\cal E}$
the complete block system of $\la\rho\ra$ consisting of $p$ blocks
of size $n/p$.  Then ${\cal C}\preceq{\cal E}$ and
$\rho^{n/q}\vert_E\in G'$ for every $E\in{\cal E}$.  Thus every $(K,H)$-generalized wreath circulant digraph is a $(L_q,M_p)$-generalized wreath circulant digraph, where $L_q$ has prime order $q$ where $q$ divides $\vert K\vert$ and $M_p$ has order $n/p$ where $p$ divides $n/\vert H\vert$.  Note that there is a unique subgroup of $\Z_n$ of prime order $q$ for each $q\vert n$, and that $M_p$ is also the unique subgroup of $\Z_n$ of order $n/p$.

As $\vert
L_q\vert = q$, we use the definition of an $(L_q, M_p)$-generalized wreath circulant digraph to conclude that $S \setminus M_p$ is a union of some subset
of the $(n-n/p)/q$ cosets of $L_q$ that are not in $M_p$.  Thus there are $2^{(n - n/p)/q}$ possible choices for the elements
of $S$ not in $M_p$. As there are at most $2^{n/p-1}$ choices for the
elements of $S$ contained in $M_p$, there are at most $2^{n/p-1}\cdot
2^{(n - n/p)/q } = 2^{n/p + n/q - n/(pq)-1}$ choices for $S$.  Summing over every possible choice of $q$ and then $p$, we see that the number of generalized wreath digraphs is bounded above by
$$\displaystyle\sum_{p|n}2^{n/p-1}\biggr(\displaystyle\sum_{q|(n/p)}2^{(n-n/p)/q}\biggr)$$
\end{proof}

\begin{cor}\label{cor-count-1}
The total number of generalized wreath circulant digraphs of order $n$ is bounded above by $(\log_2^2n)2^{n/p+n/q-n/pq-1}$, where $q$ is the smallest prime divisor of $n$ and $p$ is the smallest prime divisor of $n/q$.
\end{cor}
\begin{proof}
Note that no term in the previous summation given in Lemma \ref{numbersemiwreath} is larger than $2^{n/p+n/q-n/(pq)-1}$, where $q$ is the smallest prime divisor of $n$ and $p$ is the smallest prime divisor of $n/q$.  As the number of prime divisors of $n$ is at most $\log_2n$, we have that
$$\displaystyle\sum_{p|n}2^{n/p-1}\biggr(\displaystyle\sum_{q|(n/p)}2^{(n-n/p)/q}\biggr) \le (\log_2^2n)2^{n/p+n/q-n/(pq)-1}.$$
\end{proof}

\begin{cor}\label{graphsemiwreathhey}
The total number of generalized wreath circulant graphs of order $n$ is bounded above by $(\log_2^2n)2^{n(p+q-1)/(2pq)+1/2}$, where $q$ is the smallest prime dividing $n$, and $p$ is the smallest prime dividing $n/q$.  So the total number of generalized wreath circulant graphs of order $n$ is at most
$(\log_2^2n)2^{3n/8+1/2}$.
\end{cor}
\begin{proof}
It is straightforward using Lemma \ref{numbersemiwreath} and the fact that there are at most two elements that are self-inverse in $\Z_n$ (namely $0$ and $n/2$ if $n$ is even, and $0 \not\in S$), and at most one coset of $\Z_n/L_q$ that is self-inverse and not in $M_p$ (as $\Z_n/L_q$ is cyclic) to show that the number of generalized wreath circulant graphs of order $n$ is at most
$$\displaystyle\sum_{p|n}2^{(n/p-2)/2+1}\biggr(\displaystyle\sum_{q|(n/p)}2^{((n-n/p)/q-1)/2+1}\biggr).$$
Note that no term in this sum is larger than $2^{n(p+q-1)/(2pq)+1/2}$, where $q$ is the smallest prime dividing $n$ and $p$ is the smallest prime dividing $n/q$.  Again, as the number of prime divisors of $n$ is at most $\log_2n$,
$$\displaystyle\sum_{p|n}2^{(n/p-2)/2+1}\biggr(\displaystyle\sum_{q|(n/p)}2^{((n-n/p)/q-1)/2+1}\biggr) \le (\log_2^2n)2^{n(p+q-1)/(2pq)+1/2}.$$
Now, $(p+q-1)/pq \le 3/4$.  Hence the number of generalized wreath circulant graphs is bounded above by $(\log_2^2n)2^{3n/8+1/2}$.
\end{proof}

We now consider digraphs whose automorphism group satisfies Theorem \ref{maintool} (1).  Suppose $\Gamma$ is a circulant digraph of order $n$, and
there exist $G_1, \ldots, G_r$ such that for each $G_i$, either $G_i \cong S_{n_i}$, or
 $G_i$ contains a normal regular
cyclic group of order $n_i$. Furthermore, $r\ge
1$, $\gcd(n_i,n_j) = 1$ for $i \neq j$, and $n =
n_1n_2\cdots n_r$.
 If no $G_i\cong S_{n_i}$ with $n_i\ge 4$, then $\Aut(\Gamma)$ contains a normal regular cyclic group and $\Gamma$ is a normal circulant digraph.

\begin{defin}\label{dwdefin}
A circulant (di)graph $\Gamma$ with cyclic regular subgroup $G \cong \Z_n$ is of {\bf deleted wreath type} if
there exists some $m$ such that:
\begin{itemize}
\item $m \mid n$;
\item $\gcd(m,n/m)=1$; and
\item if $H=\la n/m\ra$ is the unique subgroup of order $m$ in $G$, then $S\cap H \in\{\emptyset, H\setminus
\{0\}\}$, and for every $g \in \la m\ra\setminus\{0\}$, $S \cap (g+H) \in \{\emptyset, \{g\},
(g+H)\setminus\{g\},  g+H\}$. (Notice that because $\gcd(m,n/m)=1$, the group $\la m \ra$ contains precisely one representative of each coset of $H$ in $G$.)
\end{itemize}
A circulant digraph is said to be of {\bf strictly deleted wreath type} if it is of deleted wreath type and is not a generalized wreath circulant.
\end{defin}

Clearly a strictly deleted wreath type circulant is a deleted wreath type circulant, but there are deleted wreath type circulants which are not strictly deleted wreath type.  For an example of the latter, consider a circulant digraph on $pqm$ vertices where $m \ge 4$ and $p$, $q$ and $m$ are relatively prime, whose connection set is $S=(\la pq\ra\setminus \{0\}) \cup (m+\la mq\ra)$.  If we let $H=\la q \ra$ and $K=\la mq\ra$, then this digraph is an $(H,K)$-generalized wreath circulant.  And if we let $H=\la pq \ra$ then $S \cap H=H\setminus\{0\}$, while for $g\in \la m \ra\setminus\{0\}$, we have $S\cap (g+H)=\{g\}$ if $g \in m+\la mq\ra$ and $S\cap (g+H)=\emptyset$ otherwise, so this digraph is of deleted wreath type.

The name deleted wreath type is chosen as these digraphs have automorphism groups that are isomorphic to the automorphism groups of deleted wreath products (see \cite{Li2005} for the definition of deleted wreath product digraphs).  We now study the automorphism groups of deleted wreath type circulant digraphs.

\begin{lem}\label{Counting V}
Let $\Gamma$ be a circulant digraph on $\Z_n$, and let $m\ge 4$ be a divisor of $n$ such that $\gcd(m,n/m)=1$. Then $\Gamma$ is of deleted wreath type with $m$ being the divisor of $n$ that satisfies the conditions of that definition, if and only if $\Aut(\Gamma)$ contains a subgroup isomorphic to $H \times S_m$ with the canonical action, for some 2-closed group $H$ with $\Z_{n/m} \le H \le S_{n/m}$.
\end{lem}

\begin{proof}
In this proof for a given $m$ satisfying $n=km$ and $\gcd(m,k)=1$, it will be convenient to consider $\Z_n = \Z_k\times\Z_m$ in the obvious fashion.  The sets $B_i=\{(i,j):j \in \Z_m\}$ for each $i\in \Z_k$ will be important.

First, suppose $\Gamma$ is of deleted wreath type with $m\ge 4$ being the divisor of $n$ that satisfies the conditions of that definition, and $n=mk$.  Using $\Z_n=\Z_k \times \Z_m$, we see that for every $i \in \Z_k\setminus\{0\}$, we have $S\cap B_i \in \{\emptyset, \{(i,0)\},B_i \setminus\{(i,0)\},B_i\}$.  Also, $S\cap B_0 \in \{\emptyset, B_0 \setminus\{(0,0)\}\}$.

Let $\cal B$ be the partition of $\Z_k \times \Z_m$ given by ${\cal B}=\cup_{i \in \Z_k}B_i$.  Let $G$ be the maximal subgroup of $\Aut(\Gamma)$ that admits $\cal B$ as a complete block system. Clearly the canonical regular cyclic subgroup isomorphic to $\Z_n=\Z_k \times\Z_m$ admits $\cal B$, so is a subgroup of $G$.

Define $H \le S_k$ to be the projection of $G$ onto the first coordinate.  Since $\Z_k \times \Z_m \le G$, clearly $\Z_k \le H$. We claim that $H\times S_m \le \Aut(\Gamma)$.  To see this, we consider the action of any element of this group, on any arc of $\Gamma$.

Let $((i_1,j_1),(i_2,j_2))$ be an arc of $\Gamma$, and let $(h,g) \in H \times S_m$.  Suppose first that $i_1=i_2$.  We have $S \cap B_0 \in \{\emptyset, B_0 \setminus\{(0,0)\}\}$, and $i_1=i_2$ forces $S\cap B_0\neq \emptyset$.  Hence the subgraph of $\Gamma$ induced by the vertices of any $B_i$ is complete, so clearly $((h(i_1),g(j_1)),(h(i_2),g(j_2)))$ is an arc since $h(i_2)=h(i_1)$.

Now suppose $i_1 \neq i_2$.  So $h(i_1)\neq h(i_2)$.  Let $i=i_2-i_1$ and let $i'=h(i_2)-h(i_1)$, with $1 \le i, i' \le k-1$. Notice that since $H$ is the projection of $G$ onto the first coordinate, there is some $g \in G$ that takes $B_{i_1}$ to $B_{h(i_1)}$ and $B_{i_2}$ to $B_{h(i_2)}$.  Hence the number of arcs in $\Gamma$ from $B_{i_1}$ to $B_{i_2}$ must be the same as the number of arcs from $B_{h(i_1)}$ to $B_{h(i_2)}$.  Now, the number of arcs in $\Gamma$ from $B_{i_1}$ to $B_{i_2}$ is $|S \cap B_i|$, while the number of arcs in $\Gamma$ from $B_{h(i_1)}$ to $B_{h(i_2)}$ is $|S\cap B_{i'}|$, so these values must be equal. Since $1 \le i, i' \le k-1$, both of these sets have the same cardinality from $\{0, 1, m-1, m\}$; in fact, since $((i_1,j_1),(i_2,j_2))$ is an arc of $\Gamma$, the cardinality cannot be $0$.  Since $m\ge 4>2$ these cardinalities are all distinct, so $S\cap B_i$ and $S\cap B_{i'}$ are uniquely determined by their cardinality.

If the cardinality is 1, then $S\cap B_i=\{(i,0)\}$ so $j_2=j_1$.  Hence $g(j_1)=g(j_2)$, and since $S\cap B_{i'}=\{(i',0)\}$, the arc $((h(i_1),g(j_1)),(h(i_2),g(j_2)))$ is in $\Gamma$.  Similarly, if the cardinality is $m-1$, then $S\cap B_i=\{B_i \setminus\{(i,0)\}\}$ so $j_2\neq j_1$.  Hence $g(j_1)\neq g(j_2)$, and since $S\cap B_{i'}=\{B_{i'} \setminus\{(i',0)\}\}$, the arc $((h(i_1),g(j_1)),(h(i_2),g(j_2)))$ is in $\Gamma$.
Finally, if the cardinality is $m$, then $S\cap B_i=B_i$, and $S\cap B_{i'}=B_{i'}$, so the arc $((h(i_1),g(j_1)),(h(i_2),g(j_2)))$ is in $\Gamma$.

We have shown that $H \times S_m \le \Aut(\Gamma)$, as desired, and that $\Z_{n/m} \le H \le S_{n/m}$.  It only remains to show that $H$ is 2-closed.
We have $H \times S_m$ admits $\mathcal B$, so by \cite[Theorem 4.11]{Wielandt1969}, so does $(H\times S_m)^{(2)}$.  Since $H \times S_m \le \Aut(\Gamma)$, we also have $(H\times S_m)^{(2)} \le \Aut(\Gamma)$ since $\Aut(\Gamma)$ is 2-closed.  By the definition of $G$, this means $(H \times S_m)^{(2)} \le G$.
By \cite[Theorem 5.1]{Cameronetal2002} we have that $(H\times S_m)^{(2)} = H^{(2)}\times (S_m)^{(2)} = H^{(2)}\times S_m$, so $H^{(2)} \times S_m \le G$.  As $H$ is projection of $G$ into the first coordinate, we conclude that $H^{(2)} = H$ and $H$ is $2$-closed.
This completes  the first direction of the biconditional.

Now we assume that $\Aut(\Gamma)$ contains a subgroup isomorphic to $H \times S_m$ with the canonical action, for some $2$-closed group $H$ with $\Z_{n/m} \le H \le S_{n/m}$.

Clearly the orbits of $\Stab_{1\times S_m}(0,0)$ are transitive on $B_i \setminus \{(i,0)\}$, and so the orbits of $\Stab_{1\times S_m}(0,0)$ on $B_i$ are $\{(i,0)\}$ and
$B_i\setminus \{(i,0)\}$. Also $1\times S_m \leq H\times S_m
 \le \Aut(\Gamma)$ implies $\Stab_{1\times S_m}(0,0)\leq
\Stab_{H\times S_m} (0,0)\leq \Stab_{\Aut(\Gamma)}(0,0)$. Thus each $S \cap B_i$
is a union of some (possibly none) of these two orbits. Hence the
only possibilities for each $S\cap B_i$ are $\emptyset, \{(i,0)\}, B_i\setminus \{(i,0)\}$ and $
B_i$ if $1\leq i\leq k-1$; and since $0 \not\in S$, $S\cap B_0$ is either $\emptyset$ or
$B_0\setminus\{(0,0)\}$.
\end{proof}

Notice that if Theorem \ref{maintool}(1) applies to $\Aut(\Gamma)$ for some circulant digraph $\Gamma$, and some $G_i \cong S_m$ where $m \ge 4$, then $\Aut(\Gamma)=H\times S_m$ for some $H$ with $\Z_{n/m}\le H \le S_{n/m}$.  Furthermore, since $\Aut(\Gamma)$ is 2-closed and the 2-closure of a direct product is the direct product of the 2-closures of the factors \cite[Theorem 5.1]{Cameronetal2002}, $H$ is 2-closed, so the above lemma tells us that $\Gamma$ is of deleted wreath type.  We also observe that provided $m\ge 4$, a deleted wreath product type circulant digraph cannot be a normal circulant digraph.

\begin{cor}\label{deletedwreathupper}
There are at most $2^{n/m+1}$ graphs $\Gamma$ and at most $2^{2n/m}$ digraphs $\Gamma$ that contain $K\times S_m$ for any choice of $K$ that is 2-closed and has $\Z_{n/m} \le K \le S_{n/m}$, where $m \ge 4$.  Equivalently, there are at most $2^{2n/m}$ digraphs of deleted wreath type, and at most $2^{n/m+1}$ graphs of deleted wreath type, for any fixed $m \ge 4$ with $m \mid n$ and $\gcd(m,n/m)=1$.
\end{cor}
\begin{proof}
A consequence of Lemma \ref{Counting V} is that there are $2\cdot 4^{n/m - 1} < 4^{n/m} = 2^{2n/m}$ digraphs of order $n$ whose automorphism group contains $K\times S_m$ for $m\ge 4$.  Note that a digraph $\Gamma$ with $\Aut(\Gamma) = K\times S_m$, $m\ge 3$, is a graph if and only if $K$ contains the map $\delta:\Z_{n/m}\to\Z_{n/m}$ given by $\delta(x) = -x$.  Then $\delta(g+H) = (-g)+H$ where $H=\la n/m\ra$, and so if $n/m$ is odd, there are at most $4^{n/(2m)} = 2^{n/m}$ graphs $\Gamma$ that contain $K\times S_m$ for any choice of $K$ that is 2-closed and has $\Z_{n/m} \le K \le S_{n/m}$.  Even if $n/m$ is even, only one nontrivial coset of $\la n/m\ra$ is fixed by $\delta$, so there are at most $2\cdot 4 \cdot 4^{(n/m-2)/2}=2^{n/m+1}$ graphs $\Gamma$ that contain $K\times S_m$ for any choice of $K$ that is 2-closed and has $\Z_{n/m} \le K \le S_{n/m}$.
\end{proof}

\section{Normal Circulants}

In this section our main focus is on determining whether or not almost all circulants that do not have automorphism groups as small as possible are normal circulants, as conjectured by the second author \cite[Conjecture 1]{Dobson2010b}.  We show that this conjecture is false for circulant digraphs of order $n$ where $n\equiv 2\ (\mod 4)$ has a fixed number of distinct prime factors (Theorem \ref{conjfalse1}).  Additionally, we correct some oversights in \cite[Theorem 3.5]{Dobson2010b}, and show that the conjecture is not true for circulant graphs of order $p$ or $p^2$, where $p$ is a safe prime, or whose order is a power of 3 (Theorem \ref{primepower}).  We also show that the conjecture is true for circulant digraphs of odd order $n$ not divisible by $9$ (Theorem \ref{normaldi}), and for circulant graphs of order $n$ if $n$ is  not a safe prime, the square of a safe prime, even, or a multiple of 9 (Theorem \ref{normal}).  We begin by showing that almost every circulant graph of order $n$ has automorphism group as small as possible.  We remark that Babai and Godsil \cite[Theorem 5.3]{BabaiG1982} have shown this to be true for abelian groups of order $n$, where $n\equiv 3\ (\mod 4)$.

Let $\ACG(n)$ be the set of all circulant graphs of order $n$, and $\Small(n)$ be the set of
all circulant graphs $\Gamma$ of order $n$ such that $\Aut(\Gamma) =
\la\rho,\iota\ra\cong D_n$, where $\iota(i) = -i$.  Thus $\Small(n)$
is the set of all circulant graphs whose automorphism groups are as
small as possible.

\begin{thrm}\label{graphsmall}
Almost all circulant graphs are in $\Small(n)$.  That is,
$$\lim_{n\to\infty}\frac{\vert \Small(n)\vert}{\vert\ACG(n)\vert} =
1.$$
\end{thrm}

\begin{proof}
By Theorem \ref{maintool} and Corollary \ref{semiwreathchar}, we have that a circulant graph of order
$n$ is either a generalized wreath circulant, or there exist $G_1, \ldots, G_r$ such that $\Aut(\Gamma)=G_1\times\ldots \times G_r$, and for each $G_i$, either $G_i \cong S_{n_i}$, or
 $G_i$ contains a normal regular
cyclic group of order $n_i$, where $r\ge
1$, $\gcd(n_i,n_j) = 1$ for $i \neq j$, and $n =
n_1n_2\cdots n_r$.  In Corollary \ref{graphsemiwreathhey}, it is shown that there are at
most $\log_2^2n\cdot 2^{3n/8+1/2}$ generalized wreath circulant graphs of
order $n$. Now, any circulant graph that is not a generalized wreath circulant graph of order
$n$ either has automorphism group which normalizes $\la\rho\ra$, or
there is some $G_i$ which is a symmetric group and $n_i\ge 4$.  If
$G_i$ is a symmetric group and $n_i\ge 4$, then $G_i$ contains a
nontrivial automorphism of $\Z_{n_i}$.  As we may assume $n$ is
arbitrarily large, we may assume that $n\not = 4$ or $6$. Then if $n_i
= n$ or $n_i = n/2$, we may choose this automorphism so that it is
not in $\la\iota\ra$. Otherwise (if $n_i <n/2$) there is some $n_j \ge 3$, and as $\gcd(n_i,n_j) = 1$ such an
automorphism of $\Z_{n_i}$ extends to an automorphism of $\Z_n$
which is not contained in $\la\iota\ra$.  In any case, we have an
automorphism $\alpha$ of $\Z_n$ contained in $\Aut(\Gamma)$ but not
in $\la\iota\ra$. Obviously, if $\la\rho\ra\tl\Aut(\Gamma)$, then
either there exists an automorphism $\alpha$ of $\Z_n$ contained in
$\Aut(\Gamma)$ but not in $\la\iota\ra$, or $\Gamma$ is in
$\Small(n)$.  Thus if $\Gamma$ is not a generalized wreath circulant graph, then
either $\Gamma$ is in $\Small(n)$ or there exists an automorphism
$\alpha$ of $\Z_n$ contained in $\Aut(\Gamma)$ but not in
$\la\iota\ra$.

Now observe that $\iota$ has at most two fixed points, and so has at
most $(n - 2)/2 + 2$ orbits. Let $\alpha\in\Aut(\Z_n)$ be such that
$\alpha\not\in\la\iota\ra$.  Observe that we may divide the orbits
of $\la\iota,\alpha\ra$ into three types:  singleton orbits, orbits
of length $2$, and orbits of length greater than $2$.  As $\la\iota\ra$
has at most $2$ singleton orbits, $\la\iota,\alpha\ra$ has at most
two singleton orbits, namely $0$ and $n/2$.  If $x\not = 0,n/2$,
then $x$ is contained in an orbit of $\la\iota\ra$ of length $2$.  If
such an $x$ is contained in an orbit of $\la\iota,\alpha\ra$ of length
$2$, then setting $\alpha(x) = ax$, $a\in\Z_n^*$, we have that
$\{x,-x\} = \{ax,-ax\}$, in which case $x = ax$ and $x$ is a fixed
point of $\alpha$ or $x = -ax$ and $x$ is a fixed point of
$\iota\alpha$.  If $x = ax$ set $\beta = \alpha$ and if $x = -ax$,
set $\beta = \iota\alpha$. Then $\la\iota,\alpha\ra =
\la\iota,\beta\ra$, and $x$ is a fixed point of $\beta$.  It is easy
to see that the set of fixed points of $\beta$, say $H(\beta)$,
forms a subgroup of $\Z_n$, and so $\vert H(\beta)\vert\le n/2$.
Thus $\la\iota,\alpha\ra$ has at at most $(n/2 - 1)/2$ orbits of
length two, and so at most $(n/2 - 1)/2 + 2$ orbits of length one or
two.  Every remaining orbit of $\la\iota,\alpha\ra$ is a union of
orbits of $\la\iota\ra$ of size $2$, and so every remaining orbit of
$\la\iota,\alpha\ra$ has length at least $4$.
Clearly, the number of orbits of $\la \iota, \alpha\ra$ is maximized if it has
2 orbits of length 1, $(n/2-1)/2$ orbits of length 2, and the remainder have length greater than 2.
In this case, there will be at
most $(n/2-1)/4 = n/8-1/4$ orbits of length greater than $2$.  We conclude
that there are at most $3n/8 + 5/4$ orbits of $\la\iota,\alpha\ra$,
and as $S$ must be a union of orbits of $\la\iota,\alpha\ra$ not including $\{0\}$, there
are at most $2^{3n/8 + 1/4}$ such circulant graphs for each
$\alpha\in\Aut(\Z_n)$, $\alpha\not = \iota$. As there are at most
$n$ (actually $\varphi(n)$ of course) automorphisms of $\Z_n$, there
are at most $n\cdot 2^{3n/8 + 1/4}$ circulant graphs that contain an
automorphism of $\Z_n$ other than $\iota$.

We have shown that there are at most $n\cdot 2^{3n/8+ 1/4}+\log_2^2n\cdot
2^{3n/8 + 1/2} < \sqrt{2}(n + \log_2^2n)2^{3n/8}$ circulant graphs of
order $n$ that are not in $\Small(n)$.  As there are $2^{(n-2)/2 + 1}
= 2^{n/2}$ circulant graphs of order $n$ if $n$ is even and
$2^{(n-1)/2}$ circulant graphs of order $n$ if $n$ is odd,
\begin{eqnarray*}
\lim_{n\to\infty}\frac{\vert \Small(n)\vert}{\vert\ACG(n)\vert}
& \ge & 1 - \lim_{n\to\infty}\frac{\sqrt{2}(n + \log_2^2n)2^{3n/8}}{2^{(n-1)/2}}\\
& = & 1 - \lim_{n\to\infty}\frac{2(n + \log_2^2n)}{2^{n/8}} =
1.
\end{eqnarray*}
\end{proof}

The above theorem clearly shows that almost all circulant graphs
are normal. In 2010, the second author proposed the following conjecture
for Cayley (di)graphs (not necessarily circulant) whose
automorphism group is not as small as possible \cite[Conjecture 1]{Dobson2010b}.

\begin{conj}\label{conjecture}
Almost every Cayley (di)graph whose automorphism group is not as
small as possible is a normal Cayley (di)graph.
\end{conj}

It is difficult to determine the automorphism group of a (di)graph,
so the main way to obtain examples of vertex-transitive graphs is to
construct them. An obvious construction is that of a Cayley
(di)graph, and the conjecture of Imrich, Lov\'{a}sz, Babai, and
Godsil says that when performing this construction, additional
automorphisms are almost never obtained. The obvious way of
constructing a Cayley (di)graph of $G$ that does not have
automorphism group as small as possible is to choose an automorphism
$\alpha$ of $G$ and make the connection set a union of orbits of
$\alpha$. The above conjecture in some sense says that this
construction almost never yields additional automorphisms other than
the ones given by the construction.

Throughout the remainder of this paper, all circulant digraphs of order $n$ whose automorphism groups are of generalized wreath, deleted wreath, and strictly deleted wreath types will be denoted by $\SW(n),\DW(n)$, and
$\SDW(n)$ respectively. The corresponding sets of all graphs whose automorphism groups are of  generalized wreath and deleted wreath type will be denoted by $\SWG(n)$ and $\DWG(n)$, respectively. Also, the sets of all digraphs that are circulants, DRR circulants, normal circulants, and non-normal circulants of order $n$ will be denoted as $\ACD(n),\DRR(n),\Nor(n)$ and $\NonNor(n)$, respectively.  The corresponding sets of all graphs that are circulants, normal circulants, and nonnormal circulants, will be denoted by $\ACG(n)$, $\NorG(n)$, and $\NonNorG(n)$, respectively.

The following lemma will prove useful in determining how many circulant (di)graphs are not normal.

\begin{lem}\label{semi-non-normal}
A circulant digraph $\Gamma$ of composite order $n$ that is a $(K,H)$-generalized wreath circulant digraph is not normal if $n$ is not divisible by $4$.
\end{lem}

\begin{proof} As usual, let $\rho:\Z_n\to\Z_n$ by $\rho(i) = i + 1\ (\mod n)$. In this proof, we will use the notation $N(n)$ for the normaliser of $\la \rho \ra$; that is, the group of permutations of $\Z_n$ given by $\{x \rightarrow ax+b: a \in \Z_n^*, b \in \Z_n\}$.  We will show that if a $(K,H)$-generalized wreath circulant is normal, then $4 \mid n$.

We may assume without loss of generality that $K$ is of prime order $p$.    Let ${\cal B}$ be the complete block system of $\la\rho\ra$ formed by the orbits of $\la\rho^m\ra$, where $\vert H\vert = n/m$. Then $\rho^{n/p}\vert_B\in\Aut(\Gamma)$ for every $B\in{\cal B}$.  Set $G = \la\rho,\rho^{n/p}\vert_B:B\in{\cal B}\ra$, and let ${\cal C}$ be the complete block system of $G$ formed by the orbits of $\la\rho^{n/p}\ra$, so that $\fix_G({\cal C}) = \la\rho^{n/p}\vert_B:B\in{\cal B}\ra$, and has order $p^{n/m}$.  Then ${\cal C}$ is also a complete block system of $N(n)$. Let $n = p_1^{a_1}p_2^{a_2}\cdots p_r^{a_r}$ be the prime power decomposition of $n$.  As $N(n) = \Pi_{i=1}^rN(p_i^{a_i})$, we see that a Sylow $p$-subgroup of $\fix_{N(n)}({\cal C})$ is a Sylow $p$-subgroup of $1_{S_{n/p^{a}}}\times N(p^{a})$, where $p = p_j$ and $a=a_j$ for some $j$.  Let ${\cal E}$ be the complete block system of $N(p^{a})$ consisting of blocks of size $p$.  Then a Sylow $p$-subgroup of $\fix_{N(p^{a})}({\cal E})$ has order at most $p^2$ as a Sylow $p$-subgroup of $N(p^{a})$ is metacyclic.  If $\Gamma$ is a normal circulant digraph, then $\la \rho \ra \triangleleft G$ since $G \le \Aut(\Gamma)$, so $G \le N(n)$. This implies that a Sylow $p$-subgroup of $\fix_G({\cal C})$ has order at most $p^2$, and so $p^{n/m} \le p^2$. Since $H>1$ we have $n>m$, so this forces $n = 2m$, and ${\cal B}$ consists of $2$ blocks. Finally, let $\delta = \rho^{n/p}\vert_B$, where $B\in{\cal B}$ with $0\in B$.  If $\Gamma$ is a normal circulant digraph, then $\gamma = \rho^{-1}\delta^{-1}\rho\delta\in\la\rho\ra$, and straightforward computations will show that $\gamma(i) = i + n/p$ if $i$ is even, while $\gamma(i) = i - n/p$ if $i$ is odd.  As $\gamma\in\la\rho\ra$, we must have that $n/p \equiv -n/p \pmod{n}$, and so $2n/p\equiv 0\pmod{n}$.  This then implies that $p = 2$ and so $4\vert n$ as required.
\end{proof}

We first show that Conjecture \ref{conjecture} is false for
circulant digraphs of order $n$, where $n\equiv 2$ (mod $4$) has a fixed number of distinct prime factors.

\begin{thrm}\label{conjfalse1}
Let $n = 2p_1^{e_1}p_2^{e_2}\cdots p_r^{e_r}$, where each $p_i$ is a distinct odd prime and $r$ is fixed.  Then

$$\lim_{n\to\infty,r{\rm\ fixed}}\frac{\vert\NonNor(n)\vert}{\vert\Nor
(n) \backslash \DRR(n)\vert} \ge \frac{1}{4(2^r - 1)}.$$
\end{thrm}

\begin{proof}
By Lemma \ref{semi-non-normal}, we have $|\NonNor(n)| \ge |\SW(n)|$.
We claim that $|\SW(n)| \ge 2^{n/2+n/(2p)-1}$, where $p$ is the smallest nontrivial divisor of $n/2$.
To see this, we construct this number of distinct generalized circulant digraphs of order $n$, as follows: $\cal B$ will be the block system formed by the orbits (cosets) of $\langle n/2 \rangle$, and $\cal C$ the block system formed by the orbits (cosets) of $\la p \ra$.  Since there are $n/p$ elements in each block of $\cal C$, there are $2^{n/p-1}$ choices for $S \cap C_0$, where $C_0$ is the block of $\cal C$ that contains 0.  Since there are $n/2-n/(2p)$ orbits (cosets) of $\la n/2 \ra$ that are not in $C_0$, there are $2^{n/2-n/(2p)}$ choices for $S-C_0$ that create a generalized circulant digraph with this choice of $\cal B$ and $\cal C$.  These $2^{n/p+n/2-n/(2p)-1}=2^{n/2+n/(2p)-1}$ generalized circulant digraphs are all distinct (though not necessarily nonisomorphic), so there are indeed at least this many distinct generalized circulant digraphs of order $n$.

Let $S(n)$ be the set of all circulant digraphs of order $n$ whose automorphism group contains a nontrivial automorphism of $\Z_n$.  Clearly then $\vert S(n)\vert\ge \vert\Nor(n) \backslash \DRR(n)\vert$.  We now seek an upper bound on $\vert S(n)\vert$.  Observe that if $\Gamma$ is a circulant digraph whose automorphism group contains a nontrivial automorphism of $\Z_n$, then $\Aut(\Gamma)$ contains a nontrivial automorphism of $\Z_n$ of prime order.

Let
$a \in \Z_n^*$ have prime order $\ell$, $a \neq 1$, and let $\alpha: \Z_n \to \Z_n$ be defined by $\alpha(i)=ai$. We first consider the case that $\alpha$ has a fixed point other than 0.
If $\alpha$ fixes a point $i$, so that  $ai\equiv i$ (mod $n$), then $(a-1)i\equiv 0$ (mod $n$). If
gcd$(i,n)=1$, then $a=1$ and $\alpha$ is the identity, a contradiction. Otherwise,
gcd$(i,n)=m$, for some non-trivial integer $m$, which clearly implies
$i\in \langle m \rangle$.
In order for $\alpha$ be an
automorphism, $a= sn/m+1$ for some $0< s <m$ must be a unit, i.e., gcd$(n,sn/m+1) =
1$. Note that $m\neq 2$ for our choice of $n$, since if $m=2$ then $s=1$, but
gcd$(n,n/2+1)\ge 2$. So $m$ must be a divisor of $n$ that is
greater than $2$ and less than $n$. Now, $\alpha$ fixes $n/m$ points $\{0, m, \cdots, (n/m-1)m\}$,
and since the order of $\alpha$ is prime ($\ell$), every non-singleton orbit of $\alpha$ has length $\ell$.
So $\alpha$ has
$n(1-1/m)/\ell$ orbits of length $\ell$, and $n/m+n/\ell-n/(m\ell)$ orbits in total.
%and since $q \ge 2$,
%at most $n(m-1)/(2m)$ non-singleton orbits.
%So the total number of orbits of
%$\alpha$ is at most $n/m+n(m-1)/(2m)=n(1/m+1/2-1/(2m))=n(1/2+1/(2m))$. Hence for any automorphism
%$\alpha$ that fixes $\la m\ra$ pointwise,
%the total number of digraphs whose automorphism groups
%contain $\alpha$, will be at most $2^{n(1/2+1/(2m))}$.
%Since $m>2$, this is maximised when $m=p$ is the smallest nontrivial divisor of $n/2$.
It will be necessary to separate out the cases where $\ell=2$ and $\ell=3$.  If $\ell=2$ then $1/m+1/\ell-1/(m\ell)=1/m+1/2-1/(2m)=1/2+1/(2m)\le (p+1)/(2p)$ since $m \ge p$ ($p$ is still the smallest nontrivial divisor of $n/2$), so if $\alpha$ has order 2 then $\alpha$ has at most $(p+1)n/(2p)$ orbits.  If $\ell=3$ then $1/m+1/\ell-1/(m\ell)=1/m+1/3-1/(3m)=1/3+2/(3m)\le (p+2)/(3p)$ since $m \ge p$, so if $\alpha$ has order 3 then $\alpha$ has at most $(p+2)n/(3p)$ orbits.  And if $\ell \ge 5$ then  $1/m+1/\ell-1/(m\ell)\le 1/m+1/5-1/(5m)=1/5+4/(5m)\le (m+4)/(5m)\le 7/15$ since $m \ge 3$, so if $\alpha$ has order greater than 3 then $\alpha$ has at most $7n/15$ orbits.

Finally,  notice that if $\alpha$ fixes only $0$, it will have 1 fixed point and $n-1$ points that are not fixed.  If $\alpha$ has order $2$ then its orbits are all of length 1 or 2, and since $n-1$ is odd, it cannot be partitioned into orbits of length 2.  So an element of order 2 must have some fixed point other than 0.  Hence if $\alpha$ fixes only $0$, it must have order at least 3, so each non-singleton orbit must have length at least 3.  Hence $\alpha$ has
at most $\lfloor(n-1)/3\rfloor<n/3$ orbits other than $\{0\}$.

From these bounds on the number of orbits of $\alpha$, we can deduce bounds on the number of normal circulant digraphs of order $n$ that admit $\alpha$ as an automorphism.
We now want to sum the upper bounds on the numbers of normal circulant digraphs of order $n$ that admit $\alpha$, over all automorphisms $\alpha$ of $\Z_n$ that have prime order.
In order to do so, we split the set $T$ of all elements of $\Z_n^*$ that have prime order, into disjoint subsets: $U$ (consisting of all elements of order 2 that have fixed points); $V$ (consisting of all elements of order 3 that have fixed points); $W$ (consisting of all elements of order 5 or greater that have fixed points) and $X$ (consisting of all elements that have no fixed points other than $0$).
Notice that $|T| \le |\Z_n^*| \le \phi(n) <n$, so the order of each set is less than $n$.  We will need slightly better estimates for $|U|$ and $|V|$; but first, observe that
$$|S(n)|\le \sum_{\alpha \in U}2^{(p+1)n/(2p)}+\sum_{\alpha\in V} 2^{(p+2)n/(3p)} + \sum_{\alpha \in W} 2^{7n/15} + \sum_{\alpha \in X}2^{n/3}.$$
Notice that $\Z_n^*=\Z_{p_1^{e_1}}^*\times \ldots \times \Z_{p_r^{e_r}}^*$ and each $\Z_{p_i^{e_i}}^*$ is cyclic, so contains a unique element of order $2$, and at most one element of order $3$.  Any element of order 2 in $\Z_n^*$ must be a product of elements of order 1 or 2 from the $\Z_{p_i}^*$, at least one of which must have order 2.  So there are $2^r-1$ elements of order 2 in $\Z_n^*$.  Similarly, there are at most $2^r-1$ elements of order 3 in $\Z_n^*$.  Thus the above sum yields
$$|S(n)| \le (2^r-1)2^{(p+1)n/(2p)}+(2^r-1)2^{(p+2)n/(3p)} + n2^{7n/15} + n2^{n/3}.$$
Since $p>1$ we have $(p+2)/(3p)<(p+1)/(2p)$, so
$$|S(n)| \le (2^r-1)2^{(p+1)n/(2p)+1}+ n2^{7n/15} + n2^{n/3}.$$
Now
\begin{eqnarray*}
\lim_{n\to\infty,r{\rm\ fixed}}\frac{\vert\NonNor(n)\vert}{\vert\Nor
(n) \backslash \DRR(n)\vert} & \ge & \lim_{n\to\infty,r{\rm\ fixed}}\frac{2^{n/2 + n/(2p)-1}}{\vert S(n)\vert}\\
& \ge & \lim_{n\to\infty,r{\rm\ fixed}}\frac{2^{n/2 + n/(2p)-1}}{(2^r-1)2^{(p+1)n/(2p)+1}+ n2^{7n/15} + n2^{n/3}}\\
& = & \lim_{n\to\infty,r{\rm\ fixed}}\frac{2^{-1}}{2(2^r - 1) + n2^{ - n/30 - n/(2p)}+n2^{-n/6-n/(2p)}}\\
& = & \frac{1}{4(2^r - 1)}.
\end{eqnarray*}
\end{proof}

A {\bf safe prime} is a prime number $p = 2q + 1$, where $q$ is also prime.

We now show that it is not true that almost all circulant graphs of order $p$ or $p^2$, where $p$ is a safe prime, or of order $3^k$, are normal.  This shows that \cite[Theorem 3.5]{Dobson2010b} is not correct.  We provide a correct statement of \cite[Theorem 3.5]{Dobson2010b} as well as point out explicitly where ``gaps" occur in the proof.  As a consequence, much of the following result is essentially the same as the proof of \cite[Theorem 3.5]{Dobson2010b}.  The entire argument is included for completeness.

\begin{thrm}\label{primepower}
Let $S = \{p,p^2:p{\rm\ is\ a\ safe\ prime}\}\cup\{3^k:k \in \N\}$, $T$ the set of all powers of odd primes, and $R = T \setminus S$.  Then
$$\lim_{n\in R,n\to\infty}\frac{\vert\NonNorG(n)\vert}{\vert \ACG(n) \setminus \Small(n)\vert} = 0.$$

\noindent Additionally, if $n\in S$, then more than one fifth of all elements of $\ACG(n) \setminus \Small(n)$ are in $\NonNorG(n)$.
\end{thrm}

\begin{proof}  Let $n=p^k$, where $p$ is an odd prime.

First suppose that $k = 1$. If $p$ is a safe prime, then $\Z_p^*$ is cyclic of order $2q$, so every element has order $2$, $q$, or $2q$.  Since a circulant graph must have $\iota$ (multiplication by $-1$) in its automorphism group, if a circulant graph of order $p$ is not in $\Small(p)$ then it must have an automorphism $\alpha$ of order $q$ or $2q$ from $\Z_p^*$ in its automorphism group.  Since the orbit of length $q$ that contains $1$ in $\Z_p^*$ does not contain $-1$, the orbits of $\la \alpha, \iota\ra$ have length $1$ (the orbit of $0$) and $2q=p-1$ (everything else).  So the graph must be either $K_p$ or its complement.  Both of these are non-normal circulants (with automorphism group $S_p$), so in this case all elements of $\ACG(n)\setminus \Small(n)$ are in $\NonNorG(n)$. (The proof of \cite[Theorem 3.5]{Dobson2010b} overlooks this case.)

Now if $p$ is not a safe prime, then $(p-1)/2$ is a composite number, say $(p-1)/2=rs$ where $1 < r\le s < (p-1)/2$.  As $p$ tends to infinity, so does $s$. Now, $\Z_p^*$ is cyclic of order $p-1$, so has an element, $\alpha$ say, of order $2r$. The action of $\alpha$ on the elements of $\Z_p$ will have $s+1$ orbits (the cosets of $\la 2r \ra$ in $\Z_p^*$, together with $0$).  Since the order of $\alpha$ is even, $-1 \in \la \alpha\ra$, so if we let $S$ be any union of these orbits, the circulant digraph on $\Z_p$ with connection set $S$ will be a graph, and since $|\alpha|>2$, this graph will not be in $\Small(p)$.  Hence $|\ACG(p)\setminus\Small(p)| \ge 2^{s+1}> 2^{\sqrt{(p-1)/2}}$.  Meanwhile, if $\Aut(\Gamma) \not< \AGL(1,p)$ then $\Aut(\Gamma) = S_p$ by \cite{Alspach1973}, and so there are only two non-normal Cayley graphs on $\Z_p$, namely $K_p$ or its complement. Clearly $2/(2^{\sqrt{(p-1)/2}})$ tends to $0$ as $p$ tends to infinity.

Now let $k\ge 2$. Through the rest of this proof, let $\alpha:\Z_n \rightarrow \Z_n$ be defined by $\alpha(i)=(p^{k-1}+1)i$. Using the binomial theorem, it is easy to see that $|\alpha|=p$.  Furthermore, $\alpha$ fixes every element of $\la p \ra$, and fixes setwise every coset of $\la p^{k-1}\ra$. Since $\alpha$ has order $p$ and $\alpha$ does not fix any element of any coset of $\la p^{k-1}\ra$ that is not in $\la p \ra$, it follows that the orbits of $\alpha$ on each coset of $\la p^{k-1} \ra$ that is not in $\la p \ra$ have length $p$, so if $\alpha \in \Aut(\Gamma)$ for some circulant graph $\Gamma$ of order $n$, then $\Gamma$ is a $(\la p^{k-1}\ra,\la p \ra)$-generalized wreath circulant digraph, and in fact by Lemma \ref{semi-non-normal}, $\Gamma$ is not normal. Conversely, if $\Gamma$ is a non-normal circulant graph of order $n$, then by Theorem \ref{maintool}, the automorphism group of a circulant graph of order $n$ either falls into category
(1) with a single factor in the direct product (since $n=p^k$ does not permit coprime factors) and consequently since it is non-normal, is complete (or empty), or category
(2) so by Corollary \ref{semiwreathchar} is a generalized wreath circulant.  Since complete and empty graphs are generalized wreath circulants, $\Gamma$ must be a generalized wreath circulant graph. It is straightforward to verify using the definition of a generalized wreath circulant, that $\alpha \in \Aut(\Gamma)$.  Notice also that if $p$ divides the order of some element $b$ of $\Z_n^*$ such that multiplication by $b$ is in $\Aut(\Gamma)$, then $\alpha \in \Aut(\Gamma)$, since $\Z_n^*$ is cyclic of order $(p-1)p^{k-1}$ so $p^{k-1}+1$ generates the unique subgroup of order $p$ in $\Z_n^*$.

Now we calculate $|\NonNorG(n)|$. As noted in the previous paragraph, if $\Gamma \in \NonNorG(n)$ then $\alpha \in \Aut(\Gamma)$, and the orbits of $\alpha$ all have length 1 or length $p$.
Now since multiplication is commutative, if $\iota$ is as usual the automorphism given by multiplication by $-1$, then $\iota$ will have a well-defined action on the orbits of $\la \alpha \ra$, and since $|\alpha|=p$ is odd, $\iota \not \in \la \alpha \ra$, so $\iota$ will exchange pairs of orbits of $\la \alpha \ra$, except the orbit $\{0\}$.  Consequently,  $\la \alpha, \iota\ra$ will have one orbit of length 1 ($\{0\}$); $(p^{k-1}-1)/2$
orbits of length $2$ (whose union is $\la p \ra\setminus\{0\}$); and $(p^k-p^{k-1})/(2p)$ orbits of length $2p$ (everything else).  So $\la \alpha, \iota\ra$ has a total of exactly $p^{k-1}+(1-p^{k-2})/2$ orbits.  Since we have shown that the non-normal circulant graphs of order $p^k$ are precisely the graphs that have $\la \alpha, \iota\ra$ in their automorphism group, there are exactly $2^{p^{k-1}+(1-p^{k-2})/2}$ non-normal circulant graphs of order $p^k$.

Now we find a lower bound for $|\ACG(n)\setminus\Small(n)|$ when $n \in R$ and $k>2$.  Since $p$ is an odd prime, $\Z_{p^k}^*$ is cyclic of order $(p-1)p^{k-1}$.  Since $p>3$, let $b$ be an element of order $p-1$ in $\Z_{p^k}^*$,
 and define $\beta:\Z_{p^k}\to\Z_{p^k}$ by $\beta(x) = bx$.  Note that $\iota\in\la\beta\ra$ since $\beta$ has even order, and $\beta\neq \iota$ since $p>3$ (the proof of \cite[Theorem 3.5]{Dobson2010b} overlooks the fact that $\beta=\iota$ when $p=3$).  Clearly, $\beta$ fixes $0$, and since the order of $\beta$ is $p-1$, every other orbit of $\beta$ has length at most $p-1$, so $\beta$ has at least $1 + (p^k - 1)/(p-1)$ orbits.  Thus there are at least $2^{1 + (p^k - 1)/(p-1)}$ circulant graphs of order $p^k$ whose automorphism group contains $\beta$, and so there are at least $2^{1 + (p^k - 1)/(p-1)}$ circulant graphs of order $p^k$ that are not in $\Small(p^k)$, $p> 3$.  Note that as $k\ge 2$, $(p^k - 1)/(p-1)\not = 1$.  Then
\begin{eqnarray*}
\lim_{p^k\to\infty}\frac{\vert \NonNorG(p^k)\vert }{\vert \ACG(p^k) \setminus \Small(p^k)\vert} & \le &  \lim_{p\to\infty}\frac{2^{p^{k-1} + (1 - p^{k-2})/2}}{2^{1 + (p^k - 1)/(p-1)}}\\
& = &  \lim_{p^k\to\infty}\frac{1}{2^{(3p^{k-2}+1)/2 + \sum_{i=0}^{k-3}p^i}}.
\end{eqnarray*}
Thus as $k\ge 3$, the result follows. (The proof of \cite[Theorem 3.5]{Dobson2010b} concludes the above limit is $1$ in all cases -- hence the gap in that theorem when $k = 2$).

We now consider the case $p=3$.  We have $\Z_{3^k}^*$ is cyclic of order $2\cdot 3^{k-1}$.  For a circulant graph $\Gamma$ of order $3^k$ to be normal but not in $\Small(3^k)$, there must be an automorphism of $\Gamma$ that corresponds to multiplying by some element, $b$ say, of $\Z_{3^k}$.  As noted previously (in the paragraph about $\alpha$), if $|b|$ is divisible by 3, then $\alpha \in \Aut(\Gamma)$ and hence $\Gamma$ is not normal.  But the only possible order for $b$ that is not divisible by 3 is 2, which corresponds to $b=-1$.  This shows that every circulant graph of order $3^k$ that is normal, is in $\Small(3^k)$, so in other words, every element of $\ACG(3^k)\setminus\Small(3^k)$ is in $\NonNorG(3^k)$.

For the remainder of the proof we suppose that $k=2$ and $p>3$.  Substituting $k=2$ into our formula for $|\NonNorG(n)|$, we conclude that $|\NonNorG(p^2)|=2^p$.

If $p$ is a safe prime, $p = 2q + 1$ with $q$ prime, then $\la \alpha \ra$ is the unique subgroup of order $p$ in $\Z_{p^2}^*$, so any subgroup of $\Z_{p^2}^*$ that contains $-1$ but does not contain $p+1$, must have even order not a multiple of $p$.  Since $\Z_{p^2}^*$ is cyclic of order $p(p-1)=2pq$, the group of order $2q$ is the only such subgroup. Call this group $B$. Then if $\Gamma$ is normal and does not have automorphism group as small as possible, then $\Aut(\Gamma) = B\cdot(\Z_{p^2})_L$.  Now, $B$ fixes $0$ and since $B$ has order $2q$ and is cyclic, the other orbits of $B$ all have length precisely $2q$ (it is not hard to show that the only elements of $\Z_{p^2}^*$ that fix anything but $0$ are $1$ and the elements of order $p$; this forces the orbit lengths of $B$ to be the order of $B$), so there are $1+(p^2-1)/2q=2+p$ orbits of $B$, and hence fewer than $2^{2+p}$ normal circulant graphs of order $p^2$ that are not in $\Small(p^2)$ (the ``fewer than" is due to the fact that some of these graphs are not normal, for example $K_{p^2}$). Hence the proportion of non-normal circulant graphs of order $p^2$ in the set of all circulant graphs of order $p^2$ that are not in $\Small(p^2)$ is more than $2^{p}/(2^p + 2^{p+2}) = 1/5$, as claimed.

Suppose now that $p$ is not a safe prime. Then there exists $b \in \Z_{p^2}^*$ of order $p-1$. Since $p$ is not a safe prime, there exists $1 <s\le r <(p-1)/2$ such that $rs=(p-1)/2$.  Let $\beta$ be the map defined by multiplication by $b$. As every non-singleton orbit of $\la\beta\ra$ has length $p-1$ (as shown for the orbits of $B$ in the preceding paragraph), every nonsingleton orbit of $\la\beta^r\ra$ has length $(p-1)/r$.  Then $\beta^r$ has $r(p+1)$ orbits not including $\{0\}$ and since $|b^r|=2s>2$, $\beta^r \neq \iota$.  We conclude that there are at least $2^{ r(p+1)}$ graphs of order $p^2$ not contained in $\Small(p^2)$.  As there are $2^p$ non-normal circulant graphs of order $p^2$ and $r>1$,
$$\lim_{p^2\to\infty}\frac{\vert \NonNorG(p^2)\vert }{\vert \ACG(p^2) \setminus \Small(p^2)\vert} \le  \lim_{p\to\infty}\frac{2^p}{2^{r(p+1)}} = 0.$$
Since $r \ge \sqrt{(p-1)/2}$, we may now conclude that
$$\lim_{n\in R,n\to\infty}\frac{\vert\NonNorG(n)\vert}{\vert \ACG(n) \setminus \Small(n)\vert} = 0.$$
\end{proof}

We now verify that Conjecture \ref{conjecture} does hold for circulant digraphs of order $n$, and also for circulant graphs of order $n$, for large families of integers.

\begin{thrm}\label{normaldi}
Let $n$ be any odd integer such that $9\nmid n$. Then almost all
circulant digraphs of order $n$ that are not DRR's are normal circulant digraphs.
\end{thrm}

\begin{proof}
It suffices to show that
\begin{eqnarray}\label{firstequation}
\lim_{n\to\infty}\frac {\vert
\NonNor(n)\vert}{\vert\ACD(n) \setminus \DRR(n) \vert}=0
\end{eqnarray}
Given any circulant digraph $\Gamma$ of order $n$, $\Aut(\Gamma)$ falls into either category (1) or (2) of Theorem \ref{maintool}. By Corollary \ref{semiwreathchar}, if $\Aut(\Gamma)$ falls into category (2), then $\Gamma$ is a generalized wreath circulant digraph.  By Lemma \ref{Counting V}, if $\Gamma$ falls into category (1) and is not normal, then $\Gamma$ is of deleted wreath type.  Hence
 $\vert \ACD(n)\vert \le \vert \Nor(n)\vert + \vert\DW(n)\vert + \vert\SW(n)\vert$, which immediately implies $\vert
\NonNor(n)\vert\le \vert\DW(n)\vert + \vert\SW(n)\vert$.  Also, a lower bound for $\vert\ACD(n) \setminus\DRR(n)\vert$ is the number of circulant graphs of order $n$, which is $2^{(n-1)/2}$.  Thus to establish (\ref{firstequation}), it suffices to show that
\begin{eqnarray*}
\lim_{n\to\infty}\frac{\vert\DW(n)\vert+\vert\SW(n)\vert}
{2^{(n-1)/2}}=0.
\end{eqnarray*}
Also note that an upper bound for $\vert\SW(n)\vert$ is given by Corollary
\ref{cor-count-1}.  We now consider an upper bound for $\vert\DW(n)\vert$.

Since $n$ is odd, we have $2n/m\le 2n/5$ for every nontrivial divisor $m\ge 4$ of $n$ (of course, in this context $m\ge 5$).  Also, $n$ is an upper bound on the number of nontrivial divisors of $n$.  By Corollary \ref{deletedwreathupper},
\begin{eqnarray*}
\lim_{n\to\infty}\frac{\vert\DW(n)\vert}{2^{(n-1)/2}}\le \lim_{n\to\infty}\frac{\sum_{m|n, m\ge 4}2^{2n/m}}{2^{(n-1)/2}}\le\lim_{n\to\infty}\frac{n\cdot 2^{2n/5}}{2^{(n-1)/2}} =  0.
\end{eqnarray*}
It thus suffices to show that $\lim_{n\to\infty}\vert\SW(n)\vert/2^{(n-1)/2} = 0$.
%Let $n$ have prime power decomposition $p_1^{a_1}p_2^{a_2}\cdots p_t^{a_t}$. Without loss of generality we assume that $3\le p_1<p_2<\cdots <p_t$ and $a_i\ge 1$ for all $i$. Note that if $p_1=3$ then $a_1=1$.
%%We set $p_k^{a_k}=\displaystyle\min \{p_i^{a_i}|1\le i \le t, p_i \neq 3\}$.  Since $p_k > 3$, we will have $p_k^{a_k} \ge 5$.
%Various cases now need to be considered depending on the value of $a_1$.

%\begin{case}
%$a_1=1$:
%\end{case}
By Corollary \ref{cor-count-1}, we have $\vert\SW(n)\vert\le \log_2^2n\cdot2^{n/p+n/q-n/(pq)-1}$, where $q$ is the smallest prime divisor of $n$ and $p$ is the smallest prime divisor of $n/q$.  Since $n$ is odd we have $q \ge 3$, and since $9 \nmid n$ we have $p \ge 5$.  If $q \ge 5$ then $1/p+1/q-1/(pq) <1/p+1/q \le 2/5$, while if $q=3$ then $1/p+1/q-1/(pq)= 2/(3p)+1/3 \le 7/15$, so we always have $1/p+1/q-1/(pq) \le 7/15$. Thus
\begin{eqnarray*}
\lim_{n\to\infty}\frac{\vert\SW(n)\vert}{2^{(n-1)/2}} \le
\lim_{n\to\infty}\frac{\log_2^2n\cdot 2^{n(1/p+1/q-1/(pq))-1}}{2^{(n-1)/2}}
 \le
\lim_{n\to\infty}\frac{\log_2^2n\cdot 2^{7n/15}}{2^{(n+1)/2}} =
\lim_{n\to\infty}\frac{\log_2^2n}{\sqrt{2}\cdot 2^{n/30}}
 = 0
\end{eqnarray*}
Note that if $9\vert n$ then $p=q=3$ and so the immediately preceding limit does not go to $0$.
%If $p_1\ge 5$, then Corollary \ref{cor-count-1} shows that the number of semiwreath
%digraphs is at most $\log_2^2n\cdot 2^{n(p_1+p_2-1)/(p_1p_2)}$.  As $p_1\ge 5$, $p_2\ge 7$, and $(p_1 + p_2 - 1)/(p_1p_2) < (p_1 + p_2)/(p_1p_2) = 1/p_1 + 1/p_2 \le 1/5 + 1/7 =  12/35$.  Hence
%
%\begin{eqnarray*}
%\lim_{n\to\infty}\frac{\vert\SW(n)\vert}{2^{(n-1)/2}}  \le
%\lim_{n\to\infty}\frac{\log_2^2n\cdot 2^{n(p_1+p_2-1)/(p_1p_2)}}{2^{(n-1)/2}} \le
%\lim_{n\to\infty}\frac{\log_2^2n\cdot 2^{12n/35}}{2^{(n-1)/2}} =
%\lim_{n\to\infty}\frac{\sqrt{2}\log_2^2n}{2^{11n/70}}
% = 0.
%\end{eqnarray*}
%
%\begin{case}
%$a_1\ge 2$:
%\end{case}
%First note that $p_1\ge 5$. By Lemma \ref{numbersemiwreath}, we have that $\vert\SW(n)\vert \le \log_2^2n\cdot 2^{n/p_1+n/p_2-n(p_1p_2)} = \log_2^2n\cdot 2^{n(2p_1-1)/p_1^2}$. As $p_1\ge 5$, $(2p_1 - 1)/p_1^2 < 2p_1/p_2^2 = 2/p_1 \le 2/5$.  Hence,
%
%\begin{eqnarray*}
%\lim_{n\to\infty}\frac{\vert\SW(n)\vert}{2^{(n-1)/2}}  \le
%\lim_{n\to\infty}\frac{\log_2^2n\cdot 2^{n(2p_1-1)/p_1^2}}{2^{(n-1)/2}} \le
%\lim_{n\to\infty}\frac{\log_2^2n\cdot 2^{2n/5}}{2^{(n-1)/2}} =
%\lim_{n\to\infty}\frac{\sqrt{2}\log_2^2n}{2^{n/10}} =
%0
%\end{eqnarray*}
\end{proof}

\begin{thrm}\label{normal}
Let $n$ be any odd integer such that $9\nmid n$, and $n$ is not a safe prime or the square of a safe prime. Then almost all
circulant graphs of order $n$ that do not have automorphism
group as small as possible are normal circulant graphs.
\end{thrm}

\begin{proof}
We need to show that
\begin{eqnarray*}
\lim_{n\to\infty,n\not\in S}\frac{\vert \mbox{NonNorG}(n)\vert}{\vert\ACG(n)\setminus \Small(n)\vert }=0,
\end{eqnarray*}
where $S=\{p, p^2: p {\rm\  is\ a\ safe\ prime}\}\cup \{n: 9 \mid n\}\cup \{n: 2 \mid n\}$. This is true if $n$ is a prime power by Theorem \ref{primepower}.  Henceforth, we assume that $n$ is not a prime power.  We may thus assume that there is a proper divisor $m$ of $n$ such that $\gcd(m,n/m) = 1$.  We assume without loss of generality that $n/m > m$, and regard $\Z_n$ as $\Z_{n/m}\times\Z_m$ in the natural way.

Given any circulant graph $\Gamma$ of order $n$, $\Aut(\Gamma)$ falls into either category (1) or (2) of Theorem \ref{maintool}. By Corollary \ref{semiwreathchar}, if $\Aut(\Gamma)$ falls into category (2), then $\Gamma$ is a generalized wreath circulant graph.  By Lemma \ref{Counting V}, if $\Gamma$ falls into category (1) and is not normal, then $\Gamma$ is of deleted wreath type.  Hence
$\vert \ACG(n)\vert \le \vert \NorG(n)\vert + \vert\DWG(n)\vert + \vert\SWG(n)\vert$, which immediately implies
$\vert
\NonNorG(n)\vert\le \vert\DWG(n)\vert + \vert\SWG(n)\vert$.

First we find a lower bound for $\vert\ACG(n)\setminus\Small(n)\vert$. Let $\Gamma\in\ACG(n)$ such that $\alpha\in \Aut(\Gamma)$ where $\alpha(i,j)=(i,-j)$ for all $(i,j)\in \Z_{n/m}\times \Z_m$.  Obviously $\alpha\notin \langle \rho,\iota\rangle $ which implies $\Gamma\not\in\Small(n)$.  Clearly if $\alpha\in\Aut(\Gamma)$, then $\la \alpha,\iota\ra \le \Aut(\Gamma)$.  It is straightforward to check that the orbits of $\la\alpha,\iota\ra$ are $\{(0,0)\}$, $\{(i,0),(-i,0)\}$, $\{(0,j),(0,-j)\}$, and $\{(i,j),(-i,j),(i,-j),(-i,-j)\}$, where $i\in \Z_{n/m} \setminus \{0\}$ and $j\in\Z_m \setminus \{0\}$.  We conclude that $\la\alpha,\iota\ra$ has
$$1 + \frac{n/m - 1}{2} + \frac{m - 1}{2} + \frac{n - n/m - m + 1}{4} = \frac{n + n/m + m + 1}{4} > \frac{n}{4}$$
orbits.  Hence there are at least $2^{n/4}$ circulant graphs of order $n$ that are not in $\Small(n)$, and we will be done if we can show that
\begin{eqnarray*}
\lim_{n\to\infty,n\not\in S}\frac{\vert \mbox{\DWG}(n)\vert+|\SWG(n)|}{2^{n/4} }=0.
\end{eqnarray*}

By Corollary \ref{deletedwreathupper}  there are there are at most $\sum_{m\vert n,m\ge 4}2^{n/m+1}$ graphs in $\DWG(n)$. Since $n$ is odd, if $m\ge 4$ and $m \mid n$ then $m \ge 5$, so $n/m \le n/5$, and $\sum_{m\vert n,m\ge 4}2^{n/m+1} \le n2^{n/5+1}$. Then
$$\lim_{n\to\infty,n\not\in S}\frac{\vert\DWG(n)\vert}{\vert \ACG(n)\setminus \Small(n)\vert}\le \frac{2n2^{n/5}}{2^{n/4}} =\frac{2n}{2^{n/20}}= 0.$$
It thus suffices to show that $\lim_{n\to\infty,n\not\in S}\vert \SWG(n)\vert/ 2^{n/4} = 0$.

By Corollary \ref{graphsemiwreathhey} there are at at most $(\log_2^2n)2^{n(p+q-1)/(2pq)+1/2}$ generalized wreath circulant graphs of order $n$, where $p$ is the smallest divisor of $n$ and $q$ is the smallest divisor of $n/p$.
%As in the previous result, we let $n$ have prime power decomposition $p_1^{a_1}p_2^{a_2}\cdots p_t^{a_t}$. Without loss of generality we assume that $3\le p_1<p_2<\cdots <p_t$ and $a_i\ge 1$ for all $i$. Note that if $p_1=3$ then $a_1=1$. We set $p_k^{a_k}=\displaystyle\min \{p_i^{a_i}|1\le i \le t\}$ if $p_1\ge
%5$, while $p_k^{a_k}=\displaystyle\min \{p_i^{a_i}|2\le i \le t\}$ if $p_1=3$ and $p_k\neq 3$.  In any case, $p_k^{a_k}\ge 5$, and so $n/m\le n/5$ for every nontrivial divisor $m\not = 3$ of $n$.  Also, $n$ is an upper bound on the number of nontrivial divisors of $n$.    We proceed by considering different possibilities for $a_1$.
%\begin{case}
%$a_1=1$:
%\end{case}
As in the proof of Theorem \ref{normaldi}, it is straightforward to show that since $n$ is odd and not divisible by $9$, $(p+q-1)/(pq) \le 7/15$.
%If $p_1=3$, then $\vert\SWG(n)\vert \le \log_2^2n2^{n(2+p_2)/(6p_2)}\le \log_2^2n2^{7n/30}$ (as $p_2\ge 5$).
Hence
$$\lim_{n\to\infty,n\not\in S}\frac{\vert \SWG(n)\vert}{2^{n/4}}\le\lim_{n\to\infty,n\not\in S}\frac{(\log_2^2n)2^{7n/30+1/2}}{2^{n/4}} =\frac{\sqrt{2}\log_2^2n}{2^{n/60}} = 0.$$
%If $p_1\ge 5$, then $p_2\ge 7$ and $\vert\SWG(n)\vert\le \log_2^2n2^{n(1/(2p_1)+ 1/(2p_2) - 1/(2p_1p_2)} \le \log_2^2n2^{n/5}$ (as $p_1\ge 5$ and $p_2\ge 7)$.  Hence
%$$\lim_{n\to\infty,n\not\in S}\frac{\vert \SWG(n)\vert}{2^{n/4}}\le\lim_{n\to\infty,n\not\in S}\frac{\log_2n2^{n/5}}{2^{n/4}} = 0.$$
%\begin{case}
%$a_1\ge 2$ :
%\end{case}
%Again note that $p_1\ge 5$. Then $\vert\SWG(n)\vert \le \log_2^2n2^{n(1/p_1 - 1/2p_1^2)}\le \log_2^2n2^{n/5 - n/2p_1^2)}$ Hence,
%$$\lim_{n\to\infty,n\not\in S}\frac{\vert \SWG(n)\vert}{2^{n/4}}\le\lim_{n\to\infty,n\not\in S}\frac{\log_2^2n2^{n/5 - n/2p_1^2)}}{2^{n/4}} = 0.$$
\end{proof}

\section{Non-normal Circulants}

By Theorem \ref{maintool}, a circulant (di)graph that is not normal is of either generalized wreath or deleted wreath type.  In this section we will consider whether or not almost all non-normal circulant (di)graphs of order $n$ are in either one of these two classes.  The short answer is ``No" and is given by the following result.

\begin{thrm}\label{counterexample}
Let $\Gamma$ be a circulant digraph of order $pq$, where $p$ and $q$ are primes and $p,q \ge 5$. Then
\begin{enumerate}
\item if $q\neq p$ then $$\frac{\vert \SW(pq)\vert}{\vert\SDW(pq)\vert} = \frac{2^{p+q-1}-2}{2^{2p-1}+2^{2q-1}-2^p-2^q-2},$$
\item if $p$ is fixed, then $\lim_{q\to\infty}\vert\SW(pq)\vert/\vert\SDW(pq)\vert = 0$,
\item if $q=p+c$ for some constant $c \ge 2$, then $\lim_{p \to \infty}
|\SW(pq)|/|\SDW(pq)|=2^c/(1+2^{2c})$
\item if $q = p$ then all non-normal circulants are generalized wreath products.
\end{enumerate}
\end{thrm}

\begin{proof}
(1): We require exact counts of $|\SW(pq)|$ and of $|\SDW(pq)|$.  First, the generalized wreath products.  When $n=pq$ a generalized wreath product will actually be a wreath product.  For a wreath product digraph with $p$ blocks of size $q$, there are $q-1$ possible elements of $S \cap \la p\ra$, and $p-1$ choices for the cosets of $\la p \ra$ to be in $S$.  Hence there are $2^{p+q-2}$ wreath product circulant digraphs with $p$ blocks of size $q$.  Similarly, there are $2^{q+p-2}$ wreath product circulant digraphs with $q$ blocks of size $p$.  The only digraphs that have both of these properties are $K_{pq}$ and its complement, each of which has been counted twice, so $|\SW(pq)|=2\cdot 2^{p+q-2}-2=2^{p+q-1}-2$.

Now we count strictly deleted wreath products.  As mentioned in the first sentence of the proof of Corollary \ref{deletedwreathupper}, there are precisely $2\cdot 4^{p-1}$ digraphs whose automorphism group contains $K \times S_q$, and $2 \cdot 4^{q-1}$ digraphs whose automorphism group contains $K' \times S_p$.  Of the first set, $2\cdot 2^{p-1}$ are wreath products (those in which $S \cap (rq+\la p \ra)$ is  chosen from $\{\emptyset, rq+\la p \ra\}$, for every $1\le r \le
p-1$).  Similarly, of the second set, $2\cdot 2^{q-1}$ are wreath products (those in which $S \cap (rp+\la q \ra)$ is  chosen from $\{\emptyset, rp+\la q \ra\}$, for every $1\le r \le
q-1$).  Finally, notice that if a digraph is counted in both the first and second sets then its automorphism group must contain $S_q \times S_p$.  Consequently, the number of elements in $S \cap (rp+\la q\ra)$ is constant over $r$, as is the number of elements in $S \cap (rq +\la p \ra)$.  Since we have already eliminated wreath products from our count, the first number must be $1$ or $p-1$, and the second must be $1$ or $q-1$.  Furthermore, if the first number is $1$ then we have $p \in S$ but $p+q \not \in S$, so the second cannot be $q-1$ (and the same holds if we exchange $p$ and $q$), so there are only 2 choices for such digraphs: that in which all of the values are $1$, which is $K_p \Box K_q$ (where $\Box$ represents the cartesian product), and its complement, in which all of the values are $p-1$ or $q-1$.  Summing up, we see that $|\SDW(pq)|=2\cdot 4^{p-1}+2\cdot 4^{q-1}-2\cdot 2^{p-1}-2\cdot 2^{q-1}-2$. The result follows.

(2): This follows from (1) by letting $q$ tend to infinity.

(3): Substituting $q=p+c$ into (1) and letting $p$ tend to infinity, we have
$$
\lim_{p \to \infty}\frac{|\SW(pq)|}{|\SDW(pq)|}=\lim_{p \to \infty}\frac{2^{c-1}-2^{1-2p}}{2^{-1}+2^{2c-1}-2^{-p}-2^{c-p}-2^{1-2p}}.
$$
Deleting the terms that tend to zero, we are left with
$$
\lim_{p \to \infty}\frac{2^{c-1}}{2^{-1}+2^{2c-1}}=\frac{2^{c}}{1+2^{2c}},
$$
as claimed.

(4): By Theorem \ref{maintool}, the automorphism group of a non-normal circulant must either fall into category (1) or category (2). If it falls into category (1) then since $n=p^2$ and the $n_i$ are coprime there can only be a single factor in the direct product, and since the circulant is non-normal, the factor must be $S_{p^2}$, so the graph is $K_{p^2}$ or its complement, which are generalized wreath circulants.  If it falls into category (2) then by Corollary \ref{semiwreathchar}, it is a generalized wreath circulant.
\end{proof}

Notice that if we choose a constant $c \ge 2$ and define $S_c=\{pq:q=p+c\}$ where $p$ and $q$ are prime, then as a consequence of Theorem \ref{counterexample}(3), since $0<2^c/(1+2^{2c})< \infty$, neither generalized wreath circulant digraphs nor strictly deleted circulant digraphs dominates in $S_c$.  Unfortunately, it is not known whether any set $S_c$ is infinite.  Essentially, we have shown that if $n=pq$ is a product of two primes, then generalized wreath products dominate amongst circulant digraphs of order $n$ if $p=q$ (in fact there are no others); neither family dominates if $p$ and $q$ are ``close" to each other, and strictly deleted wreath products dominate if one prime is much larger than the other.

We now give two infinite sets $S_1$ and $S_2$ of integers, each integer in both sets being divisible by three distinct primes.  In $S_1$, almost all non-normal circulant digraphs are of strictly deleted wreath type (and $S_1$ includes all of the square-free integers that are not divisible by $2$ or $3$). Meanwhile in $S_2$, almost all non-normal circulant digraphs are generalized wreath circulant digraphs.

\begin{thrm}\label{n=pq^2r^2}
Let $S_1= \{n\in \N\vert$ $n$ {\rm is the product of at least three primes
and $q^2\nmid n$  where $q\geq 5$ is the smallest prime divisor of} $n\}$. Then,
$$\displaystyle\lim_{n\in S_1,n\to\infty}\frac{\vert \SDW(n)
\vert}{\vert \NonNor(n)\vert} = 1$$
\end{thrm}

\begin{proof}
The first sentence of the proof of Corollary \ref{deletedwreathupper} notes that for a proper divisor $m$ of $n$ (note since $q\ge 5$ we also have $m\ge 5$), the number of digraphs $\Gamma$ with $H\times S_m\le\Aut(\Gamma)$ for some $2$-closed group $H\le S_{n/m}$ is precisely $2\cdot 4^{n/m-1}$.  The maximum number of times that a specific circulant digraph $\Gamma$ can be counted in $\displaystyle\sum_{m\vert n}2\cdot4^{n/m-1}$, is $d(n)\le n$, the number of divisors of $n$.  Thus $\vert
\DW(n)\vert \ge \displaystyle\sum_{m\vert n}2\cdot 4^{n/m-1}/n$, and so by Lemma \ref{Counting V}, $\vert\NonNor(n)\vert\ge \displaystyle\sum_{m\vert n}2\cdot 4^{n/m-1}/n$.  By Corollary
\ref{cor-count-1}, we have that $\vert \SW(n)\vert\le (\log_2^2n)2^{n/p+n/q-n/(pq)-1}$, where $q$ is the smallest prime divisor of $n$ and $p$ is the smallest prime divisor of $n/q$.  Then
\begin{eqnarray*}
\lim_{n\to\infty}\frac{\vert\SW(n)\vert}{\vert\NonNor(n)\vert} & \le & \lim_{n\to\infty}\frac{(\log_2^2n)2^{n/p+n/q - n/(pq)-1}}{\sum_{m\vert n}2\cdot 4^{n/m-1}/n}\\
& < & \lim_{n\to\infty}\frac{(\log_2^2n)2^{n/(q+2)+n/q}}{4\cdot4^{n/q-1}/n}\\
&=&\lim_{n\to\infty}\frac{n(\log_2^2n)2^{n/(q+2)}}{2^{n/q}}\\
& = & \lim_{n\to\infty}\frac{n\log_2^2n}{2^{2n/(q(q+2))}}.
\end{eqnarray*}
Since $q(q+2) < n^{2/3}$ as $q$ is the smallest prime factor of $n$, $q^2 \nmid n$, and $n$ has at least 3 prime factors, we have $n/(q(q+2)) > n^{1/3}$, so $\lim_{n\to\infty}\frac{\vert\SW(n)\vert}{\vert \NonNor(n)\vert} = 0$.  As every non-normal circulant digraph of order $n$ is either a generalized wreath or strictly deleted wreath circulant, the result follows.
\end{proof}

\begin{thrm}\label{p^2|n}
For any natural number $n$, let $p_n$ be the smallest prime divisor of $n$, and $q_n$ the smallest prime divisor of $n$ such that $q_n\neq p_n$ and $q_n^2 \nmid n$.  Let $S_2=\{n\in \N : p_n\ge 5, p_n^2\mid n,$ {\rm $n$ has at least 3 distinct prime divisors, and} $q_n>2p_n\}$. Then
$$\displaystyle\lim_{n\in S_2,n\to\infty}\frac{\vert \SW(n)\vert}{\vert \NonNor(n)\vert} = 1.$$
\end{thrm}

\begin{proof}
Let $p=p_n$. First notice that there are $2^{p-1+n/p-1}$ circulant digraphs that are wreath products $\Gamma_1 \wr \Gamma_2$ where $\Gamma_1$ has order $n/p$ and $\Gamma_2$ has order $p$: $2^{p-1}$ choices for $S \cap \la n/p \ra$ and $2^{n/p-1}$ choices for which cosets of $\la n/p \ra$ are in $S$.  All of these digraphs are distinct, so since by Lemma \ref{semi-non-normal} these are all non-normal, we have $|\NonNor(n)| \ge 2^{p+n/p-2}$.

%For any $n\in S_2$, we can express that
%$n={p_1}^{a_1}{p_2}^{a_2}\cdots {p_t}^{a_t}$, where $p_1<p_2<\cdots
%<p_t$ are all primes and $a_1\geq 2, a_i\geq 1$ for all
%$i=2,3,\cdots,t$. Let $m\vert n$ be a nontrivial divisor, with $m =
%p_1^{b_1}p_2^{b_2}\cdots p_t^{b_t}$, where $0\le b_i \le a_i$, $1\le
%i\le t$.  In order to construct a digraph $\Gamma =
%\Gamma_1\wr\Gamma_2$, where $\Gamma_2$ has order $m$, we may choose
%any subset of the unique subgroup $K$ of order $m$ of $\Z_n$, along
%with any union of the other cosets of $K$.  There are thus
%$2^{\prod_{i=1}^{t}{p_i}^{b_i}+\prod_{i=1}^{t}{p_i}^{a_i-b_i}}$ such
%wreath products.  For fixed $m$ all such digraphs are distinct, and
%so a fixed digraph $\Gamma$ can be counted at most the number of
%divisors of $n$ in $\displaystyle\sum_{0\le b_i \le
%a_i}2^{\prod_{i=1}^{t}{p_i}^{b_i}+\prod_{i=1}^{t}{p_i}^{a_i-b_i}}$.  Thus
%$$ \biggr(\displaystyle\sum_{0\le b_i \le
%a_i}2^{\prod_{i=1}^{t}{p_i}^{b_i}+\prod_{i=1}^{t}{p_i}^{a_i-b_i}}\biggr)/d(n) \le \vert\SW(n)\vert \le\vert\NonNor(n)\vert.$$
%Choosing $b_1 = a_1 - 1$ and $b_i = a_i$,
%$2\le i\le t$, we see that $\vert\NonNor(n)\vert\ge
%2^{p_1+n/p_1}/d(n)$

By Corollary \ref{deletedwreathupper}, for a proper divisor $m\ge 4$ of $n$, the number of digraphs of deleted wreath type is at most $4^{n/m}$.  Thus $\vert\DW(n)\vert\le\displaystyle\sum_{m\vert n, \gcd(m,n/m=1)}4^{n/m}$.  Let $\prod_{i=1}^t p_i^{a_i}$ be the prime decomposition of $n$, and let
$p_k^{a_k}=\displaystyle\min_{1\le i\le t}\{p_i^{a_i}\}$. Clearly
$4^{n/(p_k^{a_k})}$ is the largest term in this sum, and there are
at most $d(n)$ (the number of divisors of $n$) terms in this sum.  Thus $\vert \DW(n)\vert\le
d(n)\cdot 4^{n/(p_k^{a_k})}$.

Observe that if $a_k\ge 2$, then $p_k^{a_k}\ge 5p>2p$ since $p \ge 5$ is the smallest divisor of $n$.  Also, if $a_k = 1$, then by hypothesis $p_k \ge q_n> 2p$.  Hence $p_k^{a_k} - 2p\ge 1$ since both are integers.  Now,
\begin{eqnarray*}
\lim_{n\to\infty}\frac{\vert \DW(n)\vert}{\vert\NonNor(n)\vert} &
\le & \lim_{n\to\infty}\frac{d(n)\cdot 4^{n/({p_k}^{a_k})}}
{2^{p+n/p-2}}\\
& = & \lim_{n\to\infty}\frac{4d(n)}{2^{p+n/p-2n/({p_k}^{a_k})}}\\
& < &
\lim_{n\to\infty}\frac{4n}{2^{p+n\cdot(p_k^{a_k}-2p)/(pp_k^{a_k})}}\\
& \le &
\lim_{n\to\infty}\frac{4n}{2^{p+n/(pp_k^{a_k})}}.
\end{eqnarray*}
Since $n$ has at least 3 distinct prime divisors, there is some $j$ such that $p_j \neq p, p_k$. Now $p_j^{a_j}>p_k^{a_k}$ by our choice of $k$, and $p_j^{a_j}\ge p_j>p$, so since $n/(pp_k^{a_k}) \ge p_j^{a_j},$ we have $(n/(pp_k^{a_k}))^2 \ge pp_k^{a_k}$.  Hence $pp_k^{a_k} \le n^{2/3}$, so $n/(pp_k^{a_k}) \ge n^{1/3}$.  So the above limit is at most $$\lim_{n \to \infty} \frac{4n}{2^{p+n^{1/3}}} = 0.$$
\end{proof}

%\bibliography{References}{}

\begin{thebibliography}{10}

\bibitem{Alspach1973}
Brian Alspach, \emph{Point-symmetric graphs and digraphs of prime order and
  transitive permutation groups of prime degree}, J. Combinatorial Theory Ser.
  B \textbf{15} (1973), 12--17. \MR{MR0332553 (48 \#10880)}

\bibitem{BabaiG1982}
L{\'a}szl{\'o} Babai and Chris~D. Godsil, \emph{On the automorphism groups of
  almost all {C}ayley graphs}, European J. Combin. \textbf{3} (1982), no.~1,
  9--15. \MR{MR656006 (84d:05092)}

\bibitem{Cameronetal2002}
Peter~J. Cameron, Michael Giudici, Gareth~A. Jones, William~M. Kantor,
  Mikhail~H. Klin, Dragan Maru{\v{s}}i{\v{c}}, and Lewis~A. Nowitz,
  \emph{Transitive permutation groups without semiregular subgroups}, J. London
  Math. Soc. (2) \textbf{66} (2002), no.~2, 325--333. \MR{MR1920405
  (2003f:20001)}

\bibitem{Dobson2010b}
Edward Dobson, \emph{Asymptotic automorphism groups of {C}ayley digraphs and
  graphs of abelian groups of prime-power order}, Ars Math. Contemp. \textbf{3}
  (2010), no.~2, 200--213. \MR{2739429}

\bibitem{DobsonM2005}
Edward Dobson and Joy Morris, \emph{On automorphism groups of circulant
  digraphs of square-free order}, Discrete Math. \textbf{299} (2005), no.~1-3,
  79--98. \MR{MR2168697 (2006m:20004)}

\bibitem{EvdokimovP2002}
S.~A. Evdokimov and I.~N. Ponomarenko, \emph{Characterization of cyclotomic
  schemes and normal {S}chur rings over a cyclic group}, Algebra i Analiz
  \textbf{14} (2002), no.~2, 11--55. \MR{MR1925880 (2003h:20005)}

\bibitem{LeungM1996}
Ka~Hin Leung and Shing~Hing Man, \emph{On {S}chur rings over cyclic groups.
  {II}}, J. Algebra \textbf{183} (1996), no.~2, 273--285. \MR{MR1399027
  (98h:20009)}

\bibitem{LeungM1998}
\bysame, \emph{On {S}chur rings over cyclic groups}, Israel J. Math.
  \textbf{106} (1998), 251--267. \MR{MR1656873 (99i:20009)}

\bibitem{Li2005}
Cai~Heng Li, \emph{Permutation groups with a cyclic regular subgroup and arc
  transitive circulants}, J. Algebraic Combin. \textbf{21} (2005), no.~2,
  131--136. \MR{MR2142403 (2006b:20002)}

\bibitem{Wielandt1969}
H.~Wielandt, \emph{Permutation groups through invariant relations and invariant
  functions}, lectures given at The Ohio State University, Columbus, Ohio,
  1969.

\bibitem{Wielandt1964}
Helmut Wielandt, \emph{Finite permutation groups}, Translated from the German
  by R. Bercov, Academic Press, New York, 1964. \MR{MR0183775 (32 \#1252)}

\bibitem{Wielandt1994}
\bysame, \emph{Mathematische {W}erke/{M}athematical works. {V}ol. 1}, Walter de
  Gruyter \& Co., Berlin, 1994, Group theory, With essays on some of Wielandt's
  works by G. Betsch, B. Hartley, I. M. Isaacs, O. H. Kegel and P. M. Neumann,
  Edited and with a preface by Bertram Huppert and Hans Schneider.
  \MR{MR1272467 (95b:01025)}

\bibitem{Xu1998}
Ming-Yao Xu, \emph{Automorphism groups and isomorphisms of {C}ayley digraphs},
  Discrete Math. \textbf{182} (1998), no.~1-3, 309--319, Graph theory (Lake
  Bled, 1995). \MR{MR1603719 (98i:05096)}

\end{thebibliography}
%\bibliographystyle{amsplain}

%\end{document}
\providecommand{\bysame}{\leavevmode\hbox to3em{\hrulefill}\thinspace}
\providecommand{\MR}{\relax\ifhmode\unskip\space\fi MR }
\providecommand{\MRhref}[2]{%
  \href{http://www.ams.org/mathscinet-getitem?mr=#1}{#2}
}
\providecommand{\href}[2]{#2}

\end{document}